\documentclass[a4paper]{amsart}

\usepackage{silence}
\usepackage[utf8]{inputenc}
\usepackage{graphicx}
\usepackage{stix}
\usepackage{mathtools}
\usepackage[bf, up]{caption}
\usepackage{cite}
\usepackage{enumitem}
\usepackage[hidelinks,colorlinks=true]{hyperref}
\hypersetup{allcolors=[rgb]{0,0.31,0.62}}

\AtBeginDocument{%
}

\usepackage{algorithm}
\usepackage{algpseudocode}

\usepackage{tikz}
\usetikzlibrary{shapes.geometric}
\usepackage{tikz-cd}
\tikzcdset{
  arrow style = tikz
}
\usepackage[dvipsnames]{xcolor}



\makeatletter
\def\namedlabel#1#2{\begingroup
  #2%
  \phantomsection
  \protected@edef\@currentlabel{#2}     
  \protected@edef\@currentlabelname{#2} 
  \label{#1}\endgroup
}
\makeatother

\newcommand{\F}{\mathbb{F}} 
\newcommand{\R}{\mathbb{R}} 
\newcommand{\X}{\mathbb{X}} 
\newcommand{\Z}{\mathbb{Z}} 

\newcommand{\calI}{\mathcal{I}} 
\newcommand{\calL}{\mathcal{L}} 
\newcommand{\calO}{\mathcal{O}} 

\let\epsilon\varepsilon

\DeclareMathOperator*{\argmin}{\mathsf{argmin}}
\DeclareMathOperator{\filt}{\mathsf{Filt}}

\DeclareMathOperator{\Hom}{\mathsf{Hom}}

\DeclareMathOperator{\modd}{\mathsf{mod}}
\DeclareMathOperator{\suppp}{\mathsf{supp}}
\DeclareMathOperator{\Tor}{\mathsf{Tor}}

\newcommand{\incidence}[2]{[{#1} : {#2}]}

\newcommand{\floor}[2]{\left\lfloor \frac{#1}{#2} \right\rfloor}

\newcommand{\ceil}[2]{\left\lceil \frac{#1}{#2} \right\rceil}

\newcommand{\pairing}[2]{\langle {#1},{#2} \rangle}

\newcommand{\gen}{\text{gen}}

\newtheorem{theorem}{Theorem}[section]
\newtheorem{proposition}[theorem]{Proposition}

\newtheorem{corollary}[theorem]{Corollary}

\theoremstyle{definition}
\newtheorem{construction}{Construction}
\newtheorem{definition}[theorem]{Definition}
\newtheorem{example}[theorem]{Example}
\newtheorem{heuristic}{Heuristic}

\theoremstyle{remark}
\newtheorem{remark}[theorem]{Remark}

\algrenewcommand\algorithmicrequire{\textbf{Input:}}
\algrenewcommand\algorithmicensure{\textbf{Output:}}

\title{Lifting Cocycles: From Heuristic to Theory}
\author{Sigurd Gaukstad}
\address{Department of Mathematical Sciences, NTNU, N-7491 Trondheim, Norway}
\email{sigurd.gaukstad@ntnu.no}
\author{Mathias Karsrud Nordal}
\address{Kavli Institute for Systems Neuroscience, NTNU, N-7491 Trondheim,
  Norway}
\email{mathias.k.nordal@ntnu.no}
\author{Marius Thaule}
\address{Department of Mathematical Sciences, NTNU, N-7491 Trondheim, Norway}
\email{marius.thaule@ntnu.no}
\date{\today}

\begin{document}

%
%
\begin{abstract}
  The circular coordinates algorithm, a widely used tool in topological data
  analysis, relies on a \emph{lifting step} to convert cocycles from a prime
  field to integer coefficients. We provide a rigorous analysis of this
  procedure, establishing a criterion for its success. We give an algebraic
  method for reducing the winding number under explicit torsion hypotheses.
  The torsion and coprimality hypotheses required by the winding number
  reduction are automatic in degree $1$, because $H^1(X; \Z)$ is torsion-free.
  We extend these principles to homology, deriving analogous coefficientwise
  lifting criteria for cycles.
\end{abstract}


\maketitle

%
%
\section{Introduction}\label{sec:intro}
Topological data analysis (TDA) provides robust, scale-aware methods for
detecting the shape and connectivity of data, including methods for nonlinear
dimensionality reduction. One of these is the \emph{circular coordinates
  algorithm} proposed by de Silva, Morozov and Vejdemo-Johansson \cite{deSilva},
which uses persistent cohomology to measure the shape of a dataset and produce
circle-valued coordinates that capture its underlying topology. Their algorithm
has been applied in computational neuroscience (Rybakken et al.\
\cite{rybakken2019decoding} and Gardner et al.\ \cite{Hermansen-Nature}). In
general, cohomological feature extraction is an active area of research, with
new contributions appearing regularly, e.g., \cite{blumberg,
  generalized-penelty, paik2023circularcoordinatesdensityrobustanalysis,
  rybakken2019decoding, Schonsheck_2023, Perea2018, Toroidal}.

Perea \cite{Perea2018} developed projective coordinates from persistent
cohomology and later developed sparse circular coordinates \cite{Perea2020},
while Schonsheck and Schonsheck \cite{Schonsheck_2023} constructed spherical
coordinates from degree $2$ persistent cohomology. Together, these developments
extend cohomological coordinate constructions beyond the circle-valued setting.

A central step in the circular coordinates pipeline is the passage from a
cocycle with coefficients in a prime field to one with integer coefficients.
More precisely, after computing a class
\begin{equation*}
  [\alpha^p] \in H^1(X; \F_p)
\end{equation*}
we seek an integral class
\begin{equation*}
  [\alpha] \in H^1(X; \Z)
\end{equation*}
whose reduction modulo $p$ is $[\alpha^p]$. The existence of such a lift is
governed by the Bockstein long exact sequence. In particular, when $H^2(X; \Z)$
has no $p$-torsion, every class in $H^1(X; \F_p)$ admits an integral lift. The
original circular coordinates paper \cite{deSilva} also describes how a
coefficientwise integer lift that fails to be a cocycle can be corrected by
solving an integer linear system.

Thus, the issue considered in this paper is not the existence of integral lifts
in general. Instead, we focus on the much simpler coefficientwise lifting rule
that is commonly used in computations. For a cocycle
\begin{equation*}
  \alpha^p\in C^m(X; \F_p)
\end{equation*}
we replace each coefficient by its integer representative closest to zero. This
produces an integral cochain, but it need not produce an integral cocycle. In the
original work \cite{deSilva}, this simple lift was observed to succeed in the
examples considered there. A natural question is therefore:

\emph{When does the coefficientwise lift already satisfy the cocycle condition,
  so that no subsequent integer correction is required?}

We address this question in the present paper. We first formulate the
coefficientwise lifting procedure in a general algebraic setting and derive a
sufficient condition guaranteeing that a system of linear relations over $\F_p$
remains valid after lifting to the integers. The underlying observation is
elementary but useful: if an integer sum is known to vanish modulo $p$, and its
absolute value is strictly smaller than $p$, then the sum must vanish over $\Z$.
Applied to simplicial cocycles, this gives an explicit coefficient bound under
which the coefficientwise lift is automatically a cocycle.

We then enlarge the class of cocycles to which this criterion applies by
allowing multiplication by a nonzero scalar in $\F_p$ before lifting. Although
the original representative may contain coefficients that are too large for the
criterion, another representative of the same one-dimensional subspace may have
all coefficients sufficiently close to zero. This leads to a practical search
over nonzero elements in $\F_p$.

Our main theoretical result in this direction shows that such a rescaling is
guaranteed to exist when the characteristic is sufficiently large relative to
the number of nonzero coefficients. More precisely, for a vector $v$ in
$\F_p^n$, a pigeonhole argument gives the sufficient condition $p - 1 > k^n$ for
the existence of a nonzero scalar $r$ for which every coefficient of $rv$ has
absolute value at most
\begin{equation*}
  \left\lfloor \frac{p - 1}k\right\rfloor.
\end{equation*}
For an $m$-cocycle, taking $k = m + 2$ therefore gives an explicit sufficient
bound in terms of the size of the support of the cocycle. The bound is primarily
theoretical and can be much larger than the primes normally used in
computations. Our experiments indicate, however, that successful rescalings
often occur for substantially smaller primes. We therefore regard the
large-prime theorem as an existence result rather than as a recommendation for
choosing $p$ in practical computations.

This perspective also clarifies the computational motivation for the method.
When the coefficientwise lift fails, we may recover an integral cocycle by
solving the associated integer linear system, as in \cite{deSilva}. Our results
identify situations in which this additional correction step is unnecessary. We
do not claim a uniform computational advantage over general integer linear
algebra methods. Our aim is rather to characterize when the inexpensive
coefficientwise procedure is already sufficient and to explain mathematically
why it frequently succeeds.

A second issue arises after an integral cocycle has been obtained. Different
integral cohomology classes may encode the same underlying free direction with
different integer multiplicities. In the circular coordinate setting, this is
reflected by multiple wrapping around the circle. It is therefore useful to
determine whether the image of a class in the torsion-free quotient is
primitive.

We work with the canonical torsion-free quotient
\begin{equation*}
  H^m(X; \Z)/\Tor H^m(X; \Z),
\end{equation*}
rather than with a distinguished ``free part'' of $H^m(X; \Z)$. The winding
number of an integral cohomology class is defined as the divisibility of its
image in this free abelian quotient. We use the Kronecker pairing with an
integral homology class to obtain information about the possible prime divisors
of this winding number.

Divisibility of the image of $[\alpha]$ in the torsion-free quotient does not,
in general, imply divisibility of $[\alpha]$ in $H^m(X; \Z)$. More precisely, if
a prime $q$ divides the winding number $\Omega([\alpha])$, then $[\alpha]$ is
divisible by $q$ modulo torsion, but it need not be divisible by $q$ in the full
integral cohomology group when $q$-torsion is present. Under the coprimality
hypothesis of \autoref{prop:alg_cond}\ref{prop_alg_cond:2}, however, $q \mid
\Omega([\alpha])$ implies $\bar\rho_q([\alpha]) = 0$.

Once this vanishing has been established, no further torsion hypothesis is
required. \autoref{prop:gamma-construction} implements the integral division by
solving
\begin{equation*}
  \delta_{\F_q}f^q = \rho_q (\alpha)
\end{equation*}
over $\F_q$, choosing any integral lift $f$ of $f^q$, and defining
\begin{equation*}
  \gamma = \frac{\alpha - \delta_\Z f}q.
\end{equation*}
Then $\gamma$ is an integral cocycle and
\begin{equation*}
  \alpha = q\gamma + \delta_\Z f.
\end{equation*}

Finally, the same coefficientwise-lifting principle applies to homology.
Starting with a cycle over $\F_p$, we derive analogous conditions under which
the coefficientwise integer lift remains a cycle. The resulting bounds differ
from those for cocycles because the number of simplices incident to a given face
enters the relevant linear relations.

In the motivating case $m = 1$, the results take a particularly concrete form.
Each triangular cocycle relation involves three coefficients, so a rescaled
cocycle $r\alpha^p$ lifts coefficientwise to an integral cocycle
$\lambda_p(r\alpha^p)$ whenever
\begin{equation*}
  |\lambda_p(r\alpha^p(\sigma))| \leq \floor{p - 1}3
\end{equation*}
for every edge $\sigma$. If $n = |\suppp(\alpha^p)|$, such a rescaling is
guaranteed to exist when $p - 1 > 3^n$. After lifting, the usual least-squares
smoothing produces the unique real cocycle of minimum Euclidean $\ell^2$ norm in
the affine space of real cocycles cohomologous to the lifted cocycle. Moreover,
since $H^1(X; \Z)$ is torsion-free, the winding number reduction procedure
requires no additional torsion hypothesis in degree $1$.

\subsection*{Organization and main results}
\autoref{sec:prelim} reviews the circular coordinates construction and the
standard lifting problem. We recall the Bockstein obstruction to lifting and the
integer correction procedure described in \cite{deSilva}, thereby separating the
general existence problem from the coefficientwise heuristic studied in the
remainder of the paper.

In \autoref{sec:lifting}, we analyze the coefficientwise lift.
\autoref{thm:lifting_general} gives a general sufficient condition ensuring that
linear relations over $\F_p$ survive the lift to $\Z$, and
\autoref{cor:bound_sufficient_condition_lifting} applies this criterion to
simplicial cocycles. We then allow rescaling by nonzero elements of $\F_p$.
\autoref{thm:main} says that an appropriate rescaling exists when $p$ is
sufficiently large; \autoref{cor:main} specializes this result to cocycles.
Furthermore, we examine experimentally how often the criterion succeeds for
primes far below the theoretical bound.

\autoref{sec:finding_gen_coc} studies primitive integral representatives. We
define winding number using the quotient of integral cohomology by its torsion
subgroup and use the Kronecker pairing to constrain its prime divisors. We show
that, whenever $\bar\rho_q([\alpha]) = 0$ for a prime $q$, the division step can
be implemented by solving a coboundary equation over $\F_q$, choosing any
integral lift of the resulting cochain, and then dividing exactly by $q$. We
conclude by deriving analogous coefficientwise lifting criteria for cycles in
homology.

\subsection*{Acknowledgments}
We thank Jose A.\ Perea for helpful discussions, and we thank Andrii Bondarenko
and Danylo Radchenko for pointing out a simpler proof of \autoref{thm:main}. We
would also like to thank the referees for their input and helpful suggestions
that led to an improved version of our paper. We also acknowledge the use of
OpenAI’s ChatGPT-5.6 Sol to help improve the exposition and assist with code
development and debugging.

%
%
\section{Preliminaries}\label{sec:prelim}
In this section we review the circular coordinates construction \cite{deSilva}.
In particular, we recall the transition from prime-field cocycles to integral
cocycles. Our contribution, which will be described in the next two sections, is
centered on establishing under which conditions the simplest coefficientwise
lift is already an integral cocycle and therefore requires no correction.

We will let $X$ denote a finite simplicial complex and $p$ an odd prime
throughout the paper, and we fix orientations of the simplices of $X$. For
$x\in \Z$, we write $[x]_p$ for its residue class modulo $p$, i.e., $[x]_p\in
\F_p$ where $\F_p = \Z/p\Z$ denotes the prime field of $p$ elements.

We write
\begin{equation*}
  \rho_p\colon C^m(X; \Z) \to C^m(X; \F_p)
\end{equation*}
for coefficientwise reduction of cochains modulo $p$. Likewise, for chains we
write
\begin{equation*}
  \rho_p\colon C_m(X; \Z) \to C_m(X; \F_p)
\end{equation*}
for coefficientwise reduction modulo $p$. It will always be clear which map
$\rho_p$ is intended. Note that reduction commutes with simplicial coboundary
and boundary operators:
\begin{equation*}
  \rho_p \circ \delta_{\Z} = \delta_{\F_p} \circ \rho_p \qquad \text{and}
  \qquad \rho_p \circ \partial_\Z = \partial_{\F_p} \circ \rho_p.
\end{equation*}
We denote the induced maps on cohomology and homology by
\begin{equation*}
  \bar{\rho}_p\colon H^m(X; \Z) \to H^m(X; \F_p) \qquad \text{and} \qquad
  \bar{\rho}_p \colon H_m(X; \Z) \to H_m(X; \F_p),
\end{equation*}
respectively.

\begin{definition}\label{def:unique_integer}
  For $x\in \F_p$, let
  \begin{equation*}
    \lambda_p\colon \F_p \to \left\{-\frac{p - 1}2,\dots, -1, 0, 1,\dots,
      \frac{p - 1}2\right\} \subseteq \Z
  \end{equation*}
  where $\lambda_p(x)$ is the unique integer satisfying $[\lambda_p(x)]_p = x$.
  We call $\lambda_p(x)$ the \emph{balanced integer representative} of $x$.
\end{definition}

We extend $\lambda_p$ coefficientwise to vectors, chains and cochains.
Note that, by definition, $\rho_p(\lambda_p(\alpha^p)) = \alpha^p$.

\begin{definition}\label{def:integral_cochain_cocycle_lift}
  If $\alpha^p\in C^m(X; \F_p)$, then an \emph{integral cochain lift} of
  $\alpha^p$ is a cochain $\alpha\in C^m(X; \Z)$ such that
  \begin{equation*}
    \rho_p(\alpha) = \alpha^p.
  \end{equation*}
  If both $\alpha^p$ and $\alpha$ are cocycles, we call $\alpha$ an
  \emph{integral cocycle lift} of $\alpha^p$.
\end{definition}

The distinction between an integral cochain lift and an integral cocycle lift
will be important later on. A coefficientwise lift always exists but it need not
preserve the cocycle condition.

\subsection{Motivation}
Many constructions in topological data analysis associate coordinates to
persistent cohomology classes. In particular, circular coordinates
\cite{deSilva}, projective coordinates \cite{Perea2018} and spherical
coordinates \cite{Schonsheck_2023}, are motivated by the relationship between
cohomology and maps to Eilenberg--Mac~Lane spaces.

We recall a special case of Brown's representability theorem.

\begin{theorem}\label{thm:brown_representability}
  Let $X$ be a pointed connected CW complex, let $m$ be an integer greater than
  or equal to $1$ and let $G$ be an abelian group. Then there exists a pointed
  CW complex $K(G,m)$, called an Eilenberg--Mac~Lane space, and a natural
  bijection of sets
  \begin{equation*}
    H^m(X;G) \cong \left[X, K(G,m)\right].
  \end{equation*}
  In particular, $H^m(-; G)$ and $[-, K(G,m)]$ are naturally isomorphic as
  contravariant functors from the category of pointed connected $CW$ complexes
  to the category of sets.
\end{theorem}

\begin{remark}
  Circular coordinates correspond to the case $m = 1$, $G = \Z$ and $K(\Z, 1)
  \simeq S^1$. Thus an integral first cohomology class determines, up to
  homotopy, a circle-valued map.
\end{remark}

\subsection{The circular coordinates algorithm}
Let $\X\subseteq \R^D$ be a point cloud. The circular coordinates algorithm
\cite{deSilva} consists of five steps:
\begin{description}[leftmargin = 2\parindent, style = multiline]
\item[\namedlabel{it:1}{(1)}] Represent the data as a filtered simplicial
  complex $\filt(\X)$. Common choices include the \emph{\v{C}ech} complex and
  its computationally efficient approximation, the \emph{Vietoris--Rips}
  complex.
\item[\namedlabel{it:2}{(2)}] Compute persistent cohomology and identify a
  topologically significant cohomology class $[\alpha^{p}] \in
  H^1(\filt_{\epsilon}(\X); \F_{p})$ at an appropriate scale $\epsilon$.
\item[\namedlabel{it:3}{(3)}] Obtain an integral cocycle $\alpha$ representing
  an integral cohomology class whose reduction is the chosen prime-field class,
  $[\alpha^p]$. (This is the \emph{lifting step}.)
\item[\namedlabel{it:4}{(4)}] Replace $\alpha$ with an improved real cocycle
  $\tilde{\alpha}$ cohomologous to $\alpha$.
\item[\namedlabel{it:5}{(5)}] Use $\tilde{\alpha}$ to construct a circle-valued
  function
  \begin{equation*}
    \vartheta_{X}(\tilde{\alpha})\colon \filt_{\epsilon}(\X) \to \R/\Z \cong
    S^1,
  \end{equation*}
  that reveals $[\tilde{\alpha}]$.
\end{description}

We are particularly interested in \ref{it:3} in this paper, and therefore review
it more carefully.

For a cocycle $\alpha^p\in C^m(X; \F_p)$, there are two related but distinct
questions:
\begin{description}[leftmargin = 2\parindent, style = multiline]
\item[\namedlabel{q:1}{(1)}] Does $[\alpha^p]\in H^m(X; \F_p)$ admit an integral
  cohomology lift?
\item[\namedlabel{q:2}{(2)}] Does $\lambda_p(\alpha^p)$ already satisfy the
  integral cocycle condition?
\end{description}
The first question is governed by the Bockstein long exact sequence, while the
second question is the subject of the next section.

\begin{theorem}[Bockstein long exact sequence]\label{thm:bockstein_les}
  The short exact sequence
  \begin{equation*} 
    \begin{tikzcd}[column sep = scriptsize]
      {0} & {\Z} & {\Z} & {\F_p} & {0}
      \arrow[from=1-1, to=1-2]
      \arrow["p", from=1-2, to=1-3]
      \arrow[from=1-3, to=1-4]
      \arrow[from=1-4, to=1-5]
    \end{tikzcd}
  \end{equation*}
  induces a long exact sequence in cohomology
  \begin{equation*}
    \begin{tikzcd}[column sep = scriptsize]
      {\cdots} \arrow{r} &
      {H^m(X; \Z)} \arrow{r}{p} &
      {H^m(X; \Z)} \arrow{r}{\bar{\rho}_p}
      \arrow[phantom, ""{coordinate, name=B}]{d} &
      {H^m(X; \F_p)} \arrow[swap,
      rounded corners,
      to path={
        -- ([xshift=2ex]\tikztostart.east)
        |- (B) [near end]\tikztonodes
        -| ([xshift=-2ex]\tikztotarget.west)
        -- (\tikztotarget)
      }
      ]{dll}[at end]{\beta} \\
      &
      {H^{m + 1}(X; \Z)} \arrow{r}{p} &
      {H^{m + 1}(X; \Z)} \arrow{r} &
      {\cdots}
    \end{tikzcd}
  \end{equation*}
  where the connecting homomorphism $\beta$ is called the \emph{Bockstein
    homomorphism}.

  In particular, $[\alpha^p]$ admits an integral cohomology lift if and only if
  $\beta([\alpha^p]) = 0$.
\end{theorem}

Note that if $H^{m + 1}(X; \Z)$ contains no $p$-torsion, multiplication by $p$
on $H^{m + 1}(X; \Z)$ is injective, and so, $\bar{\rho}_p \colon H^m(X; \Z) \to
H^m(X; \F_p)$ is surjective by exactness. Thus, in the absence of $p$-torsion in
degree $m + 1$, every prime-field cohomology class has an integral lift.

We now describe the naive (or simple or trivial) lifting heuristic.

\begin{heuristic}\label{heu:lifting_to_integer}
  Let $\alpha^p\in C^m\left(X; \F_{p}\right)$ be an $m$-cocycle. We construct an
  integral cochain $\alpha$ by choosing each coefficient $\alpha_\sigma$ from the
  range
  \begin{equation*}
    \left\{-\frac{p-1}{2}, \ldots, -1, 0, 1, \ldots, \frac{p-1}{2}\right\},
  \end{equation*}
  such that $\alpha_\sigma \equiv \alpha_\sigma^p \modd p$. In particular,
  $\alpha$ is given by
  \begin{equation*}
    \alpha = \lambda_p(\alpha^p).
  \end{equation*}
\end{heuristic}

The following result generalizes the argument of \cite[Section 2.4]{deSilva},
and it describes how to obtain an integral cocycle representative starting from
a coefficientwise lift.

\begin{proposition}\label{prop:constructing_integer_cocycle}
  Let $\alpha^{p} \in C^m(X; \F_{p})$ be a cocycle and suppose that $H^{m +
    1}(X; \Z)$ has no $p$-torsion. Then there exists an integral cocycle $\alpha
  \in C^m(X;\Z)$ such that $\rho_p(\alpha) = \alpha^p$.

  More explicitly, such a cocycle can be obtained by starting from the balanced
  cochain lift
  \begin{equation*}
    \tilde{\alpha} \coloneq \lambda_p(\alpha^p)
  \end{equation*}
  and, if necessary, correcting it by an integral cochain.
\end{proposition}

\begin{proof}
  We will construct $\tilde{\alpha}$ by \autoref{heu:lifting_to_integer} and obtain
  $\alpha$ after a correction step.

  By definition, $\rho_p(\tilde{\alpha}) = \alpha^p$. If
  $\delta_\Z\tilde{\alpha} = 0$, then $\tilde{\alpha}$ itself is an integral
  cocycle lift and there is nothing more to prove.

  Suppose instead that $\delta_\Z\tilde{\alpha} \neq 0$. Since $\alpha^p$ is a
  cocycle and reduction commutes with the coboundary operator,
  \begin{equation*}
    \rho_p(\delta_\Z\tilde{\alpha}) = \delta_{\F_p}\rho_p(\tilde{\alpha}) =
    \delta_{\F_p}\alpha^p = 0.
  \end{equation*}
  Thus every coefficient of $\delta_\Z\tilde{\alpha}$ is divisible by $p$, and
  so there exists a cochain $\eta\in C^{m + 1}(X; \Z)$ such that
  \begin{equation}\label{eq:p_eta}
    \delta_\Z \tilde{\alpha} = p\eta
  \end{equation}
  and so, $\delta_\Z^2\tilde{\alpha} = \delta_\Z p\eta = p \delta_\Z \eta = 0$.
  Since $C^{m + 2}(X; \Z)$ is free abelian, $\delta_\Z \eta = 0$. Hence, $\eta$
  defines a cohomology class $[\eta]\in H^{m + 1}(X; \Z)$ with $p[\eta] =
  [p\eta] = 0$. Moreover, $\eta$ is a coboundary as $H^{m + 1}(X; \Z)$ has no
  $p$-torsion.

  Hence, there exists $\xi\in C^m(X; \Z)$ such that $\delta_\Z\xi = \eta$. Now
  define
  \begin{equation*}
    \alpha \coloneq \tilde{\alpha} - p\xi.
  \end{equation*}
  Then $\delta_\Z \alpha = \delta_\Z \tilde{\alpha} - p\delta_\Z \xi = p\eta -
  p\eta = 0$. Thus $\alpha$ is an integral cocycle. Moreover, $\rho_p(\alpha) =
  \rho_p(\tilde{\alpha}) = \alpha^p$, since the correction term $p\xi$ reduces
  to $0$ modulo $p$. Therefore $\alpha$ is an integral cocycle lift of
  $\alpha^p$.
\end{proof}

\begin{remark}
  \autoref{prop:constructing_integer_cocycle} shows that in the absence of
  $p$-torsion in $H^{m + 1}(X; \Z)$, the existence of an integral cocycle lift
  is guaranteed by the Bockstein long exact sequence. Moreover, the balanced
  lift $\lambda_p(\alpha^p)$ may or may not be itself a cocycle. If it is not,
  the correction in \autoref{prop:constructing_integer_cocycle} requires solving
  the equation $\delta_\Z \xi = \eta$, which may be solved using Smith normal
  form or other methods for integer linear equations. For example, DREiMac
  implements the correction step using integer linear programming
  \cite{DREiMac}.

  We do not claim that this correction step is generally infeasible, nor that
  the methods in the next section are uniformly more efficient. We are concerned
  with identifying which explicit conditions ensure that the correction step is
  unnecessary, i.e., $\lambda_p(\alpha^p)$ is already a cocycle.

  If $H^{m + 1}(X; \Z)$ does contain $p$-torsion,
  \begin{equation*}
    \bar{\rho}_p \colon H^m(X; \Z) \to H^m(X; \F_p)
  \end{equation*}
  need not be surjective. Note that a particular class $[\alpha^p]\in H^m(X;
  \F_p)$ admits an integral lift precisely when $\beta([\alpha^p]) = 0$.
\end{remark}

For the remainder of this section we consider the case of circular coordinates,
i.e., $m = 1$. The following construction corresponds to
\cite[Proposition~1]{deSilva}.

\begin{construction}[circle-valued function from cocycle]\label{const:1}
  Let $X$ be a pointed connected simplicial complex and let $\alpha \in
  C^1(X;\Z)$ be a cocycle. We construct a continuous function
  \begin{equation*}
    \vartheta_X(\alpha)\colon |X| \to \R/\Z \cong S^1,
  \end{equation*}
  inductively, by mapping each vertex to $0$, and each edge $ab$ around the
  circle with winding number $\alpha(ab)$, and extending continuously over
  triangles as the cocycle condition makes each boundary loop null-homotopic,
  and then over higher simplices as $\pi_n(S^1) = 0$ for $n > 1$.
\end{construction}

Using \autoref{const:1} we can build a circle-valued function on $|X|$ that
reveals $[\alpha]$. However, since all vertices are mapped to the base point of
the circle, all the circular variation is captured in the edges and higher-order
simplices. In particular, circular variation is not captured in the
$0$-skeleton, which is ultimately what we aim for, since the $0$-skeleton equals
the input data $\X$. We would therefore like to pick another class
representative. This is the \emph{smoothing step} from \cite{deSilva} and is
described by the following proposition.

\begin{proposition}[{\cite[Proposition~4]{deSilva}}]\label{prop:harmonic_ls}
  Let $\iota^\#\colon C^1(X; \Z) \to C^1(X; \R)$ be the induced map on cochains
  by the canonical inclusion $\iota\colon \Z \hookrightarrow \R$, and let
  $\alpha \in C^1(X; \Z)$ be a cocycle. The least-squares minimization problem
  \begin{equation*} 
    \argmin_{\tilde{\alpha} \in C^1(X; \R)} \left\{ \|\tilde{\alpha}\|^2 \mid \exists
      f \in C^0(X; \R), \;\tilde{\alpha} = \iota^\#(\alpha) + \delta_\R f \right\}
  \end{equation*}
  has a unique solution $\tilde{\alpha}$. This solution is characterized by
  $\delta_\R^*\tilde{\alpha} = 0$, where $\delta_\R^*$ is the adjoint of
  $\delta_\R$ with respect to the canonical inner products on $C^0(X;\R)$ and
  $C^1(X;\R)$.
\end{proposition}

\begin{proposition}[{\cite[Proposition~2]{deSilva}}]\label{prop:real-cocycle-construction}
  Let $\alpha\in C^1(X; \Z)$ be a cocycle, and let $\tilde{\alpha}\in C^1(X;
  \R)$ be a cocycle that is cohomologous to $\iota^\#(\alpha)$, so that
  \begin{equation*}
    \tilde{\alpha} = \iota^\#(\alpha) + \delta_\R f
  \end{equation*}
  for some $f\in C^0(X; \R)$. Then there is a continuous function
  \begin{equation*}
    \vartheta_X(\tilde{\alpha})\colon |X| \to \R / \Z \cong S^1,
  \end{equation*}
  mapping each edge $ab \in |X|^{(1)}$ to an interval of signed length
  $\tilde{\alpha}\left(ab\right)$. More precisely, define
  \begin{equation*}
    \vartheta_X(\tilde{\alpha})(v) = f(v) \modd \Z
  \end{equation*}
  on vertices, and map each edge $ab$ linearly with signed increment
  $\tilde{\alpha}(ab)$.
\end{proposition}

%
%
\section{Lifting to Integer Coefficients}\label{sec:lifting}
The previous section highlighted two facts about the lifting problem. The
Bockstein long exact sequence guarantees, under the appropriate torsion
hypothesis, the existence of integral cocycle lifts. Moreover, if
$\lambda_p(\alpha^p)$ is not an integral cocycle,
\autoref{prop:constructing_integer_cocycle} gives a correction procedure. We now
seek explicit sufficient conditions under which $\lambda_p(\alpha^p)$ is already
a cocycle, so that the correction step in
\autoref{prop:constructing_integer_cocycle} is unnecessary.

\subsection{A no-wrap-around criterion}\label{sec:no_wrap}
For $d\geq 1$, we write
\begin{equation*}
  \calL_{p,d} \coloneq \left\{ -\floor{p - 1}d, \dots, \floor{p - 1}d\right\}
  \subseteq \Z.
\end{equation*}
For $m$-cocycles, we write
\begin{equation*}
  \calL_m^p \coloneq \calL_{p, m + 2}.
\end{equation*}

Note that if $a_1, a_2, \dots, a_d\in \Z$ all lie in $\calL_{p,d}$, then
\begin{equation*}
  \left| \sum_{j = 1}^d a_j \right| \leq d \floor{p - 1}d \leq p - 1.
\end{equation*}
Hence, if the sum is divisible by $p$, it must be zero.

\begin{theorem}\label{thm:lifting_general}
  Let $v = (v_1, v_2, \dots, v_n)\in \F_p^n$, and let $\calI$ be a finite index
  set. For every $i\in \calI$, let $I_i\subseteq \{1, 2, \dots, n\}$ with
  $I_i\neq \emptyset$, and for $j\in I_i$, choose $\epsilon_{ij}\in \{-1,1\}$.
  Suppose that
  \begin{equation}\label{eq:eps_ij}
    \sum_{j\in I_i} \epsilon_{ij} v_j = 0
  \end{equation}
  in $\F_p$ for every $i\in \calI$, and suppose further that there exists $r\in
  \F_p^\ast$ such that for every j occurring in at least one of these relations,
  \begin{equation}\label{eq:min}
    |\lambda_p(rv_j)| \leq \min_{\substack{i\in \calI\\ j\in I_i}}
    \floor{p - 1}{|I_i|}.
  \end{equation}
  Then
  \begin{equation*}
    \tilde{v}_r \coloneq \lambda_p(rv)\in \Z^n
  \end{equation*}
  satisfies the corresponding integral relations:
  \begin{equation}\label{eq:eps_ij_tilde_v}
    \sum_{j\in I_i} \epsilon_{ij} (\tilde{v}_r)_j = 0
  \end{equation}
  for every $i\in \calI$.
\end{theorem}

\begin{proof}
  Fix $i\in \calI$ and define
  \begin{equation*}
    S_i \coloneq \sum_{j\in I_i} \epsilon_{ij} \lambda_p(rv_j) \in \Z.
  \end{equation*}
  Each $\lambda_p(rv_j)\in \Z$ represents $rv_j$ modulo $p$. Thus, by
  \autoref{eq:eps_ij},
  \begin{equation*}
    [S_i]_p = \sum_{j\in I_i} \epsilon_{ij} rv_j = r\sum_{j\in I_i}
    \epsilon_{ij} v_j = 0
  \end{equation*}
  in $\F_p$, and so, $S_i\in p\Z$.

  From \autoref{eq:min}, we have
  \begin{equation*}
    |S_i| \leq \sum_{j\in I_i} |\lambda_p(rv_j)| \leq |I_i| \floor{p - 1}{|I_i|}
    \leq p - 1 < p.
  \end{equation*}
  As the only multiple of $p$ whose absolute value is strictly smaller than $p$
  is $0$, we have $S_i = 0$, and since $i$ was arbitrary, all relations in
  \autoref{eq:eps_ij_tilde_v} hold.
\end{proof}

We now specialize \autoref{thm:lifting_general} to the case of simplicial
cocycles. In this special case, the relations indexed by $\calI$ are the cocycle
relations associated with the $(m + 1)$-simplices of $X$.

Let $\alpha^p\in C^m(X; \F_p)$ be a cocycle. Note that, for every oriented $(m +
1)$-simplex $\tau$, we have
\begin{equation}\label{eq:incidence}
  0 = (\delta_{\F_p} \alpha^p)(\tau) = \sum_{
    \substack{
      \sigma\subset \tau\\
      \dim \sigma = m
    }
  } \incidence{\tau}{\sigma} \alpha^p(\sigma)
\end{equation}
where $\incidence{\tau}{\sigma}\in \{-1,1\}$ is the incidence number. Note also
that, since $\tau$ has exactly $m + 2$ $m$-dimensional faces, every relation in
\autoref{eq:incidence} contains $m + 2$ terms.

The following corollary follows by applying \autoref{thm:lifting_general} to the
cocycle relations of \autoref{eq:incidence}.

\begin{corollary}\label{cor:bound_sufficient_condition_lifting}
  Let $\alpha^p\in C^m(X; \F_p)$ be a cocycle, and suppose that there exists
  $r\in \F_p^\ast$ such that $\lambda_p(r\alpha^p(\sigma))\in \calL_m^p$ for
  every oriented $m$-simplex $\sigma$. Then
  \begin{equation*}
    \tilde{\alpha}_r \coloneq \lambda_p(r\alpha^p)
  \end{equation*}
  is an integral $m$-cocycle.

  Moreover, if $s\in \Z$ satisfies $[s]_p = r^{-1}$, then
  \begin{equation*}
    \alpha \coloneq s\tilde{\alpha}_r
  \end{equation*}
  is an integral cocycle lift of the original prime-field cocycle
  $\rho_p(\alpha) = \alpha^p$.
\end{corollary}

\begin{proof}
  Since every relation has $m + 2$ terms, we have
  \begin{equation*}
    |\lambda_p(r\alpha^p(\sigma))| \leq \floor{p - 1}{m + 2}
  \end{equation*}
  where we have used the assumption that $\lambda_p(r\alpha^p(\sigma))\in
  \calL_m^p$. Thus $\delta_\Z \tilde{\alpha}_r = 0$, and so, $\tilde{\alpha}_r$
  is an integral cocycle. Moreover, multiplication by $s$ preserves the cocycle
  condition, so $\alpha = s\tilde{\alpha}_r$ is also an integral cocycle.

  Furthermore, $\rho_p(\tilde{\alpha}_r) = r\alpha^p$. Since $[s]_p = r^{-1}$,
  we have
  \begin{equation*}
    \rho_p(\alpha) = \rho_p(s\tilde{\alpha}_r) = [s]_p \rho_p
    (\tilde{\alpha}_r) = r^{-1} (r\alpha^p) = \alpha^p.
  \end{equation*}
  Thus $\alpha$ is an integral cocycle lift of $\alpha^p$.
\end{proof}

\begin{example}\label{ex:no_cocycle}
  Let $p = 7$ and consider the cocycle represented by $\alpha^7 = (3, 4, 1)\in
  \F_7^3$ with cocycle relation $3 - 4 + 1 = 0$ in $\F_7$. The balanced
  coefficientwise lift of \autoref{heu:lifting_to_integer} is
  \begin{equation*}
    \lambda_7(\alpha^7) = (3, -3, 1)
  \end{equation*}
  and the corresponding integer relation is $3 - (-3) + 1 = 7$. Thus
  $\lambda_7(\alpha^7)$ is \emph{not} an integral cocycle.

  Now take $r = 2$. Then $2\alpha^7 = (6, 1, 2)$, and
  \begin{equation*}
    \lambda_7(2\alpha^7) = (-1, 1, 2).
  \end{equation*}
  The lifted relation is then $-1 - 1 + 2 = 0$. Thus $\tilde{\alpha}_2 =
  \lambda_7(2\alpha_7)$ \emph{is} an integral cocycle. Since $[4]_7 = 2^{-1}$,
  the integral cocycle $4\tilde{\alpha}_2 = (-4, 4, 8)$ satisfies
  $\rho_7(4\tilde{\alpha}_2) = \alpha^7$. \autoref{fig:exmple_scaling_lifting}
  illustrates the procedure.
  \begin{figure}[htbp]
    \centering
    \begin{tikzpicture}[thick]
      \node at (0,2){$\F_7$};
      \node at (4.5,2){$\Z$};
      \node at (9,2){$\Z$};

      \node[name = trekant1,%
      regular polygon,%
      regular polygon sides = 3,%
      draw,%
      fill = black!25,%
      minimum width = 5em] at (0,0) {};
      \node[above] at (trekant1.corner 1){$c$};
      \node[below left] at (trekant1.corner 2){$a$};
      \node[below right] at (trekant1.corner 3){$b$};
      \node[left, xshift = -0.1cm] at (trekant1.side 1){$4$};
      \node[below] at (trekant1.side 2){$1$};
      \node[right, xshift = 0.1cm] at (trekant1.side 3){$3$};

      \node[name = trekant2,%
      regular polygon,%
      regular polygon sides = 3,%
      draw,%
      fill = black!25,%
      minimum width = 5em] at (4.5,0) {};
      \node[above] at (trekant2.corner 1){$c$};
      \node[below left] at (trekant2.corner 2){$a$};
      \node[below right] at (trekant2.corner 3){$b$};
      \node[left, xshift = -0.1cm] at (trekant2.side 1){$-3$};
      \node[below] at (trekant2.side 2){$1$};
      \node[right, xshift = 0.1cm] at (trekant2.side 3){$3$};

      \node[name = trekant3,%
      regular polygon,%
      regular polygon sides = 3,%
      draw,%
      fill = black!25,%
      minimum width = 5em] at (0,-5) {};
      \node[above] at (trekant3.corner 1){$c$};
      \node[below left] at (trekant3.corner 2){$a$};
      \node[below right] at (trekant3.corner 3){$b$};
      \node[left, xshift = -0.1cm] at (trekant3.side 1){$1$};
      \node[below] at (trekant3.side 2){$2$};
      \node[right, xshift = 0.1cm] at (trekant3.side 3){$6$};

      \node[name = trekant4,%
      regular polygon,%
      regular polygon sides = 3,%
      draw,%
      fill = black!25,%
      minimum width = 5em] at (4.5,-5) {};
      \node[above] at (trekant4.corner 1){$c$};
      \node[below left] at (trekant4.corner 2){$a$};
      \node[below right] at (trekant4.corner 3){$b$};
      \node[left, xshift = -0.1cm] at (trekant4.side 1){$1$};
      \node[below] at (trekant4.side 2){$2$};
      \node[right, xshift = 0.1cm] at (trekant4.side 3){$-1$};

      \node[name = trekant5,%
      regular polygon,%
      regular polygon sides = 3,%
      draw,%
      fill = black!25,%
      minimum width = 5em] at (9,-5) {};
      \node[above] at (trekant5.corner 1){$c$};
      \node[below left] at (trekant5.corner 2){$a$};
      \node[below right] at (trekant5.corner 3){$b$};
      \node[left, xshift = -0.1cm] at (trekant5.side 1){$4$};
      \node[below] at (trekant5.side 2){$8$};
      \node[right, xshift = 0.1cm] at (trekant5.side 3){$-4$};

      \draw[shorten <= 0.9cm,%
      shorten >= 0.9cm,%
      ->] (trekant1.side 3) -- node[auto]{Heuristic 1} (trekant2.side 1);
      \draw[shorten <= 0.9cm,%
      shorten >= 0.9cm,%
      ->] (trekant3.side 3) -- node[auto]{Heuristic 1} (trekant4.side 1);
      \draw[shorten <= 0.9cm,%
      shorten >= 0.9cm,%
      ->] (trekant4.side 3) -- node[auto]{Scale by $4$} (trekant5.side 1);
      \draw[shorten <= 0.75cm,%
      shorten >= 0.75cm,%
      ->] (trekant1.side 2) -- node[auto, swap]{Scale by $2$} (trekant3.corner 1);
      \draw[shorten <= 0.9cm,%
      shorten >= 0.9cm,%
      red,%
      ->] (trekant1.corner 3) -- node[auto,text = black]{Lift of $\alpha^p$} (trekant5.side 1);
    \end{tikzpicture}
    \caption{An example where applying \autoref{heu:lifting_to_integer} does not
      yield a cocycle, but when we allow for scaling in the prime field before
      lifting, \autoref{heu:lifting_to_integer} does yield a cocycle.}
    \label{fig:exmple_scaling_lifting}
  \end{figure}
\end{example}

The rescaling step can also be useful for cocycles obtained in persistent
cohomology computations. In \autoref{fig:trefoil_fixing_coordinates}, a cocycle
over $\F_{47}$ has coefficients in $\{1, 23, 24, 46\}$. After multiplication by
$2$, the coefficients are $\{2, 46, 1, 45\}$ whose balanced representatives are
$\{2, -1, 1, -2\}$. Thus multiplication by a scalar can move the coefficients
into the range required by \autoref{cor:bound_sufficient_condition_lifting} even
when the original representative does not satisfy the criterion.

\begin{figure}[htbp]
  \centering
  \includegraphics[width=1\linewidth]{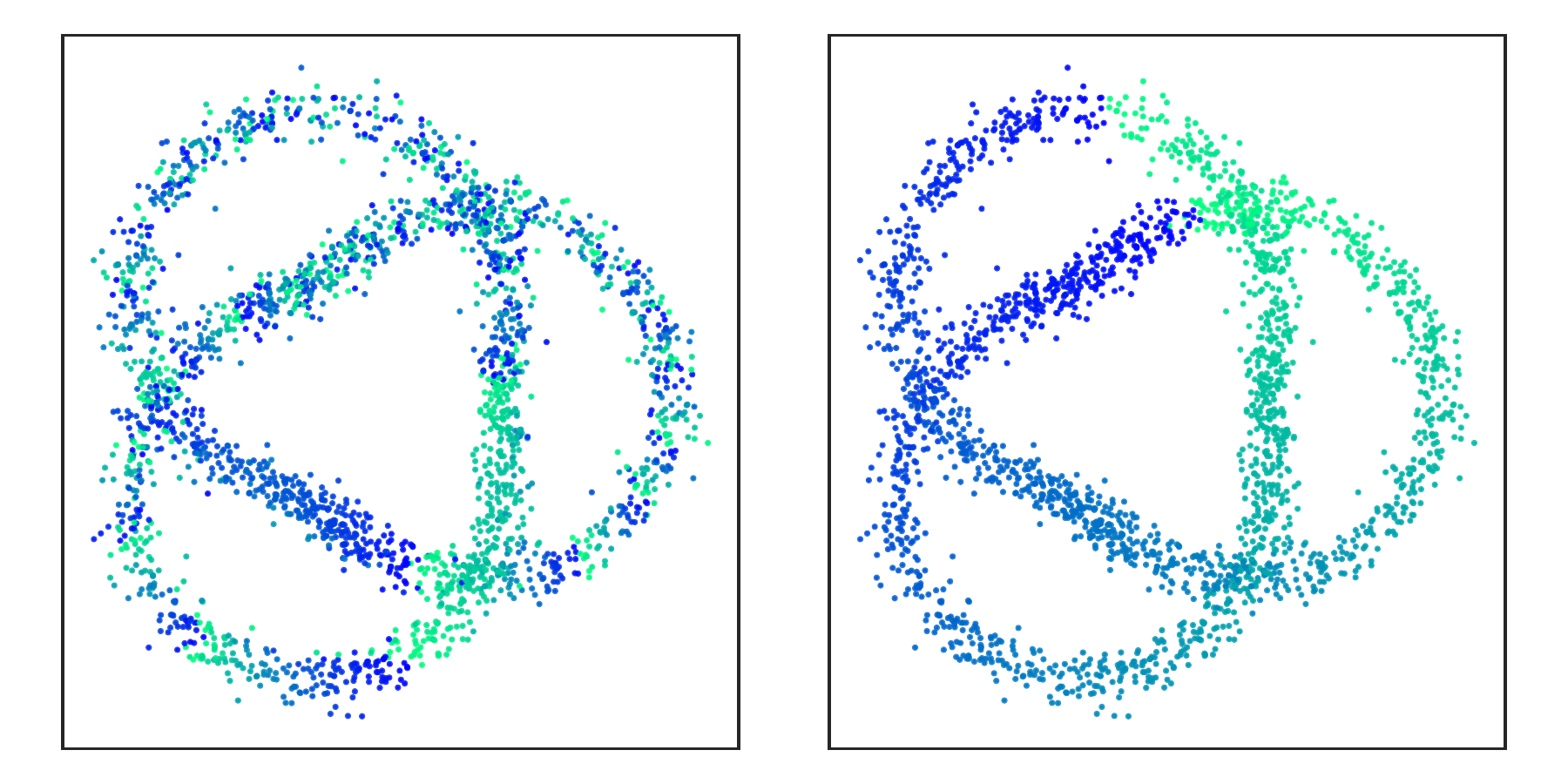}
  \caption{The balanced lift has nonzero integral coboundary and therefore does
    not define a continuous map $|X| \to S^1$ by \autoref{const:1}. The left
    panel shows only the vertex values returned by the subsequent least-squares
    calculation. The right panel shows that by scaling the cocycle by $2$ before
    lifting, the coefficientwise lift is the integral cocycle
    $\lambda_{47}(2\alpha^{47})$.}
  \label{fig:trefoil_fixing_coordinates}
\end{figure}

All persistent homology and cohomology computations were performed using
\texttt{Ripser.py} \cite{ripserpy} in Python and \texttt{Ripserer.jl}
\cite{Cufar2020} in Julia. Both implementations use the Ripser persistence
algorithm \cite{Bauer}.\sloppy

\subsection{A large-prime guarantee}\label{sec:large_prime}
\autoref{cor:bound_sufficient_condition_lifting} gives a sufficient condition
for a particular prime-field cocycle and a particular scalar. A natural
question is then: must such an appropriate scalar exist? For a fixed prime $p$,
the answer is in general no. However, if $p$ is sufficiently large relative to
the number of coefficients, a pigeonhole argument gives a uniform guarantee.

\begin{theorem}\label{thm:main}
  Let $n$ and $k$ be integers that are greater than or equal to $1$, and let $p$
  be an odd prime satisfying $p - 1 > k^n$. Then, for every $v = (v_1, v_2,
  \dots, v_n)\in \F_p^n$, there exists $r\in \F_p^\ast$ such that
  \begin{equation*}
    |\lambda_p(rv_j)| \leq \floor{p - 1}k
  \end{equation*}
  for every $j = 1, 2, \dots, n$.
\end{theorem}

\begin{proof}
  If $v = 0$, there is nothing to prove. Assume therefore that $v\neq 0$.

  For the purpose of the proof, identify $\F_p$ with the set of integer
  representatives
  \begin{equation*}
    \{0, 1, \dots, p - 1\}.
  \end{equation*}
  For $\ell = 1, 2, \dots, k$, we define the interval
  \begin{equation*}
    J_\ell \coloneq \left\{ \ceil{(\ell - 1)p}k, \ceil{(\ell - 1)p}k + 1,
      \dots, \ceil{\ell p}k - 1\right\}.
  \end{equation*}
  Then $\{0, 1, \dots, p - 1\} = J_1 \sqcup J_2 \sqcup \dots \sqcup J_k$. Note
  that each $J_\ell$ contains either $\lfloor p/k \rfloor$ or $\lceil p/k
  \rceil$ integers, and so, the diameter of each $J_\ell$ is at most
  \begin{equation}\label{eq:diam}
    \ceil{p}k - 1 = \floor{p - 1}k.
  \end{equation}
  Taking Cartesian products, we obtain a partition of $\F_p^n$ into $k^n$ boxes
  of the form
  \begin{equation*}
    J_{\ell_1} \times J_{\ell_2} \times \cdots \times J_{\ell_n}, \qquad \ell_i\in
    \{1, 2, \dots, k\}.
  \end{equation*}
  Now consider the $p - 1$ scalar multiples $v, 2v, \dots, (p - 1)v$. Since, by
  assumption, $p - 1 > k^n$, two of these vectors must lie in the same box,
  e.g., $r_1 v$ and $r_2 v$. Since $v\neq 0$ and $r_1\neq r_2$, $r\coloneq r_1 -
  r_2\in \F_p^\ast$.

  Fix a coordinate $j$. Let $a_j, b_j\in \{0, 1, \dots, p - 1\}$ be the chosen
  integer representatives of $r_1 v_j$ and $r_2 v_j$, respectively. Since $r_1
  v$ and $r_2 v$ lie in the same box, $a_j$ and $b_j$ belong to the same
  interval $J_{\ell_j}$. Thus, by \autoref{eq:diam}, we have
  \begin{equation*}
    |a_j - b_j| \leq \floor{p - 1}k.
  \end{equation*}
  Moreover, $a_j - b_j$ is an integer representative modulo $p$ of $r_1 v_j -
  r_2 v_j = rv_j$. Since $\lambda_p(rv_j)$ is the integer representative of
  $rv_j$ having smallest absolute value, we obtain
  \begin{equation*}
    |\lambda_p(rv_j)| \leq |a_j - b_j| \leq \floor{p - 1}k.
  \end{equation*}
  Finally, since $j$ was arbitrary,
  \begin{equation*}
    |\lambda_p(rv_j)| \leq \floor{p - 1}k
  \end{equation*}
  for every $j = 1, 2, \dots, n$ as required. This finishes the proof.
\end{proof}

\begin{corollary}\label{cor:main}
  Let $\alpha^p\in C^m(X; \F_p)$ be a cocycle and define the support of
  $\alpha^p$ as
  \begin{equation*}
    \suppp(\alpha^p) \coloneq \{\sigma\in X^{(m)} \mid \alpha^p(\sigma) \neq 0\}.
  \end{equation*}
  Let $n \coloneq |\suppp(\alpha^p)|$. If $p - 1 > (m + 2)^n$ then there exists
  $r\in \F_p^\ast$ such that $\lambda_p(r\alpha^p)$ is an integral $m$-cocycle.
  Furthermore, choosing any $s\in \Z$ with $[s]_p = r^{-1}$ gives an integral
  cocycle lift $s\lambda_p(r\alpha^p)$ of $\alpha^p$.
\end{corollary}

\begin{proof}
  If $\alpha^p = 0$, the conclusion is immediate. Assume therefore that
  $\alpha^p \neq 0$. By applying \autoref{thm:main} to the vector consisting of
  the $n$ nonzero coefficients of $\alpha^p$, with $k = m + 2$, there exists
  $r\in \F_p^\ast$ such that
  \begin{equation*}
    |\lambda_p(r\alpha^p(\sigma))| \leq \floor{p - 1}{m + 2}
  \end{equation*}
  for every $\sigma\in \suppp(\alpha^p)$. Outside of the support the coefficient
  is zero, so the same bound holds trivially. The statement then follows by
  applying \autoref{cor:bound_sufficient_condition_lifting}.
\end{proof}

\begin{figure}[htbp]
  \centering
  \includegraphics[width=1.0\linewidth]{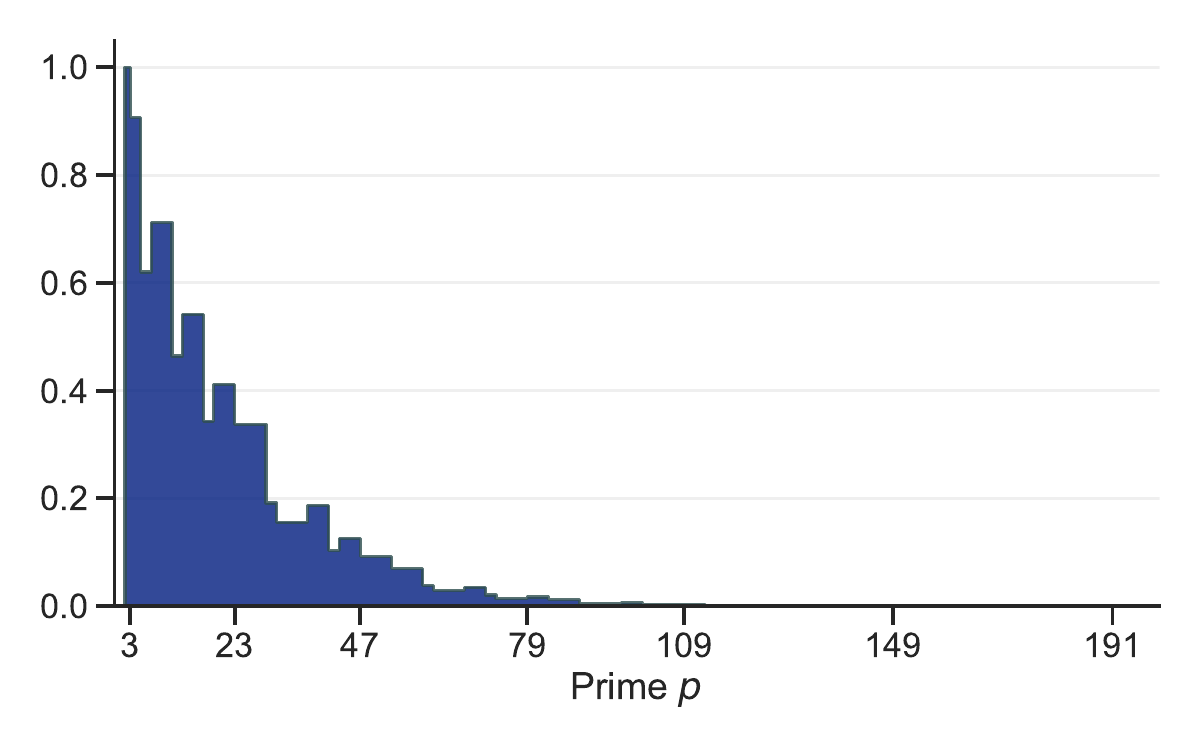}
  \caption{Let $n = 6$. For each prime $p$ in the range $[3,200]$ we sample up
    to $1{,}000{,}000$ lines in $\F^n_p$ and record the proportion of lines that
    do not contain a vector whose balanced lift lies in $(\calL_1^p)^n$.}
  \label{fig:non_liftable_are_sparse}
\end{figure}

\begin{remark}
  For a fixed prime $p$, we can search directly for a scalar $r$ satisfying the
  criterion in \autoref{cor:bound_sufficient_condition_lifting}. If $n =
  |\suppp(\alpha^p)|$, then a straightforward exhaustive search tests at most
  $n$ coefficients for each of the $p - 1$ nonzero scalars, requiring $\calO(n(p
  - 1))$ coefficient checks.

  The bound
  \begin{equation*}
    p - 1 > (m + 2)^n
  \end{equation*}
  in \autoref{cor:main} is of a different nature. It is a worst-case existence
  guarantee and may require a prime that is exponentially large in the support
  size. We therefore do not regard \autoref{cor:main} as a practical recommendation
  for the choice of $p$. Note that \autoref{cor:main} does not by itself
  establish that the rescaling procedure is computationally preferable to
  solving the integer linear system in
  \autoref{prop:constructing_integer_cocycle}.

  Empirically, the sufficient condition in
  \autoref{cor:bound_sufficient_condition_lifting} may succeed for primes far
  below the theoretical bound. \autoref{fig:non_liftable_are_sparse} illustrates
  this phenomenon for the examples considered here. These experiments measure
  the failure frequency of the coefficient criterion. They
  should not be interpreted as a general complexity comparison with integer
  linear algebra methods.

  Finally, if $\tilde{\alpha}_r = \lambda_p(r\alpha^p)$ is the integral cocycle
  produced by \autoref{cor:bound_sufficient_condition_lifting}, then multiplying
  by an integer $s$ satisfying $[s]_p = r^{-1}$ produces an integral cocycle
  whose reduction is the original prime-field representative $\alpha^p$. Note
  that for circular coordinates, multiplication by $s$ may introduce additional
  divisibility in the integral cohomology class. The problem of detecting and
  removing such divisibility is considered in the next section.
\end{remark}

%
%
\section{Finding a cocycle with winding number \texorpdfstring{$1$}{1}}\label{sec:finding_gen_coc}
The winding number of an integral cocycle will be reflected in the circular
coordinates, cf.\ \autoref{fig:order_and_circcoords_1}. Hence, we would like the
lifted integral cocycle to have winding number $1$. Moreover, in
high-dimensional data it can be difficult to determine the winding number of an
integral cocycle based on its circular coordinates as seen in
\autoref{fig:generators_matters}. Hence, a method for determining the winding
number would be preferred. Under the hypotheses described below, we present an
algebraic method for determining the winding number and reducing a cocycle to
winding number $1$.

\begin{figure}[htbp]
  \centering
  \includegraphics[width=1.0\linewidth]{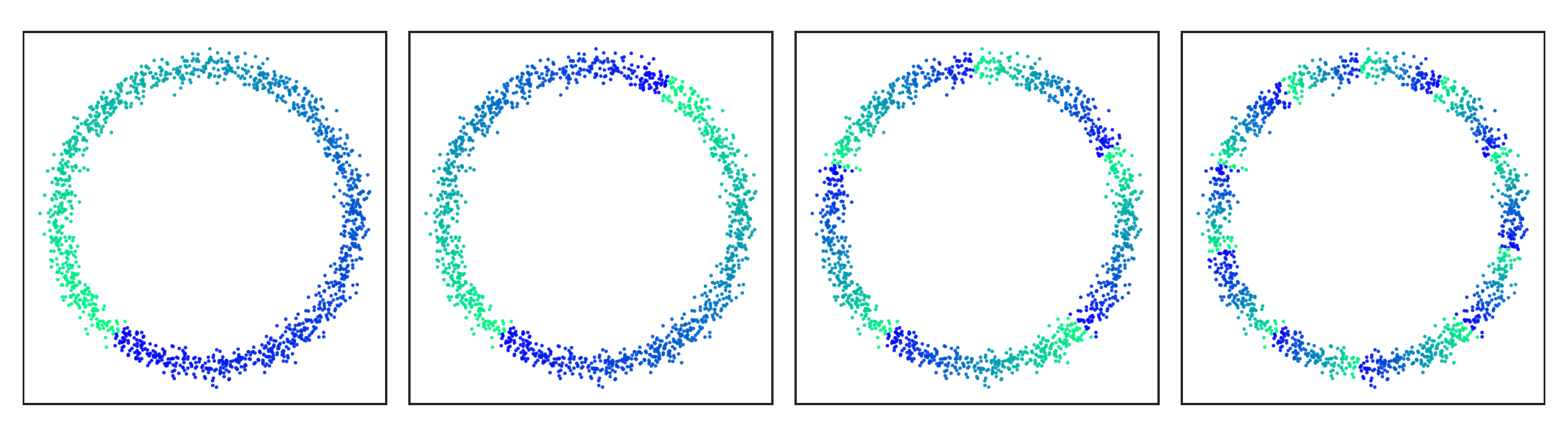}
  \caption{Circular coordinates constructed from cocycles of winding numbers $1,
    2, 5$, and $10$ from left to right.}
  \label{fig:order_and_circcoords_1}
\end{figure}

Our method of determining and reducing the winding number of a cocycle $\alpha$
requires a cycle $\beta$ such that the Kronecker pairing $\langle
[\alpha],[\beta] \rangle$ is nonzero. By the universal coefficient theorem, the
evaluation map
\begin{equation*}
  H^m(X; \Z) \to \Hom(H_m(X; \Z), \Z)
\end{equation*}
has torsion kernel. Hence, if $[\alpha]$ has nonzero image in $H^m(X; \Z)/\Tor
H^m(X; \Z)$, there exists $[\beta]\in H_m(X; \Z)$ such that
$\pairing{[\alpha]}{[\beta]} \neq 0$. In degree $1$, an integral representative
$\beta$ can be constructed directly from a spanning forest, as described below.
In higher degrees, one may compute integral homology representatives by integer
linear algebra applied to the chain complex; field-valued representatives
require an additional lifting step. In \autoref{subsec:lifting_homology} we
derive an analogous coefficientwise lifting criterion for homology cycles.

\subsection{Determining the winding number of a cocycle}\label{subsec:order_cocycle}
A finitely generated abelian group $A$ need not contain a canonical subgroup
that can be called its free part. What is canonical, however, is the
torsion-free quotient
\begin{equation*}
  A/\Tor(A),
\end{equation*}
which is a finitely generated free abelian group. We write $\pi_T \colon A\to
A/\Tor(A)$ for the canonical quotient map.

\begin{definition}\label{def:winding}
  Let $X$ be a simplicial complex and $m \geq 1$. Let $[\alpha]\in H^m(X; \Z)$
  have a nonzero image under
  \begin{equation*}
    \pi_T \colon H^m(X; \Z) \to H^m(X; \Z)/\Tor H^m(X; \Z),
  \end{equation*}
  i.e., there is a unique positive integer $w$ and a \emph{primitive} element
  \begin{equation*}
    [e]\in H^m(X; \Z)/\Tor H^m(X; \Z),
  \end{equation*}
  meaning $[e]$ is not divisible by any integer greater than $1$, such that
  $\pi_T([\alpha]) = w[e]$. We call
  \begin{equation*}
    \Omega([\alpha]) \coloneq w
  \end{equation*}
  \emph{the winding number} of $[\alpha]$.
\end{definition}

Equivalently, after choosing any basis $H^m(X; \Z)/\Tor H^m(X; \Z) \cong \Z^n$,
if
\begin{equation*}
  \pi_T([\alpha]) = (g_1, g_2, \dots, g_n),
\end{equation*}
then
\begin{equation*}
  \Omega([\alpha]) = \gcd(|g_1|, |g_2|, \dots, |g_n|).
\end{equation*}
In other words, the definition does not depend on the choice of basis.

\begin{remark}\label{rem:winding}
  We note that in the case $m = 1$, which is the case relevant to circular
  coordinates, $H^1(X; \Z)$ is torsion-free, and so,
  \begin{equation*}
    H^1(X; \Z) = H^1(X; \Z)/\Tor H^1(X; \Z).
  \end{equation*}
  Thus \autoref{def:winding} reduces to ordinary divisibility of the integral
  cohomology class. Moreover, \autoref{def:winding} also agrees with the usual
  geometric interpretation of winding number. From the fact that
  \begin{equation*}
    H^1(X; \Z) \cong [X, S^1],
  \end{equation*}
  we get that if $[\alpha] = w[\gamma]$, and $\nu_\alpha, \nu_\gamma \colon X\to
  S^1$ represent the corresponding homotopy classes, then $\nu_\alpha \simeq
  \mu_w \circ \nu_\gamma$ where $\mu_w \colon S^1 \to S^1$, given by $\mu_w(z) =
  z^w$, is the degree $w$ map. In other words, the algebraic divisibility $w$
  measures precisely the number of times the resulting circle-valued coordinate
  wraps around $S^1$.
\end{remark}

In order to determine the winding number of an integral cocycle algebraically,
we will use the Kronecker pairing as follows.

\begin{proposition}\label{prop:order_of_cocycle}
  Let $[\alpha]\in H^m(X; \Z)$ have nonzero image in $H^m(X; \Z)/\Tor H^m(X;
  \Z)$. Then for every $[\beta]\in H_m(X; \Z)$, the Kronecker pairing $\langle
  [\alpha],[\beta] \rangle$ is divisible by the winding number $\Omega([\alpha])$.
  In particular, if $\langle [\alpha],[\beta] \rangle \neq 0$, then every prime
  divisor of $\Omega([\alpha])$ is a prime divisor of $\langle [\alpha],[\beta]
  \rangle$.
\end{proposition}

\begin{proof}
  Let $\pi_T([\alpha]) = \Omega([\alpha]) [e]$, where $[e]$ is primitive. The
  Kronecker pairing vanishes on torsion cohomology classes: if $[\tau]$ is
  torsion and $n[\tau] = 0$, then
  \begin{equation*}
    n \pairing{[\tau]}{[\beta]} = \pairing{n[\tau]}{[\beta]} = 0.
  \end{equation*}
  Since the Kronecker pairing takes values in $\Z$, this implies that
  $\pairing{[\tau]}{[\beta]} = 0$.

  Hence, the Kronecker pairing factors through the torsion-free quotient, and
  so,
  \begin{equation*}
    \pairing{[\alpha]}{[\beta]} = \pairing{\pi_T([\alpha])}{[\beta]} =
    \pairing{\Omega([\alpha])[e]}{[\beta]} = \Omega([\alpha]) \pairing{[e]}{[\beta]}
  \end{equation*}
  where we have used the same notation, $\pairing{\cdot}{\cdot}$, for the
  induced Kronecker pairing on $H^m(X; \Z)/\Tor H^m(X; \Z)$. In other words,
  $\Omega([\alpha])$ divides the Kronecker pairing. \sloppy
\end{proof}

Let $\bar{\rho}_q \colon H^m(X; \Z) \to H^m(X; \F_q)$ be the coefficient-change
map induced by reduction modulo a prime $q$. The Bockstein long exact sequence
associated to
\begin{equation*}
  \begin{tikzcd}[column sep = scriptsize]
    {0} & {\Z} & {\Z} & {\F_q} & {0}
    \arrow[from=1-1, to=1-2]
    \arrow["q", from=1-2, to=1-3]
    \arrow[from=1-3, to=1-4]
    \arrow[from=1-4, to=1-5]
  \end{tikzcd}
\end{equation*}
gives
\begin{equation}\label{eq:ker_rho_q}
  \ker \bar{\rho}_q = q H^m(X; \Z).
\end{equation}
Thus $\bar{\rho}_q([\alpha]) = 0$ detects actual divisibility of $[\alpha]$ in
the full integral cohomology group, and not just divisibility of its image
modulo torsion.

We will establish an algorithm for finding a cocycle with winding number $1$. To
do this we will need the following result.

\begin{proposition}\label{prop:alg_cond}
  Suppose that $[\alpha]\in H^m(X; \Z)$ has winding number $w$. Let $q$ be a
  prime.
  \begin{description}[leftmargin = 2\parindent, style = multiline]
  \item[\namedlabel{prop_alg_cond:1}{(i)}] If $\bar{\rho}_q([\alpha]) = 0$, then
    $w$ is divisible by $q$.
  \item[\namedlabel{prop_alg_cond:2}{(ii)}] If $w$ is divisible by $q$ and
    $|{\Tor H^m(X; \Z)}|$ is \emph{not} divisible by $q$, then
    $\bar{\rho}_q([\alpha]) = 0$.
  \end{description}

  Thus if $H^m(X; \Z)$ is torsion-free, then the winding number of $\alpha$,
  $w$, is divisible by $q$ if and only if $\bar{\rho}_q([\alpha]) = 0$.
\end{proposition}

\begin{proof}
  (i) Assume that $\bar{\rho}_q([\alpha]) = 0$. By \autoref{eq:ker_rho_q},
  $[\alpha] = q[\gamma]$ for some $[\gamma]\in H^m(X; \Z)$. Passing to the
  torsion-free quotient then gives
  \begin{equation*}
    \pi_T([\alpha]) = q\pi_T([\gamma]).
  \end{equation*}
  Thus $\pi_T([\alpha])$ is divisible by $q$, and so is its winding number.

  (ii) Assume that $w$ is divisible by $q$. Then the image of $[\alpha]$ in
  $H^m(X; \Z)/\Tor H^m(X; \Z)$ is divisible by $q$. Hence, there is some
  $[\rho]\in H^m(X; \Z)$ such that $[\alpha] - q[\rho] \in \Tor H^m(X; \Z)$. Let
  $[\tau] = [\alpha] - q[\rho]$.

  Since $|{\Tor H^m(X; \Z)}|$ is not divisible by $q$, multiplication by $q$ is
  an automorphism of the (finite) group $\Tor H^m(X; \Z)$. Thus there is some
  $[\sigma]\in \Tor H^m(X; \Z)$ such that $[\tau] = q[\sigma]$, and so,
  \begin{equation*}
    [\alpha] = q[\rho] + q[\sigma] = q([\rho] + [\sigma]).
  \end{equation*}
  Hence, $[\alpha]\in qH^m(X; \Z)$, and so, by \autoref{eq:ker_rho_q},
  \begin{equation*}
    \bar{\rho}_q([\alpha]) = 0.
  \end{equation*}
  This finishes the proof.
\end{proof}

\begin{remark}\label{rem:referee}
  We cannot omit the hypothesis concerning torsion as shown by the following
  example pointed out to us by a referee. In the group $\Z \oplus \F_3$, the
  elements
  \begin{equation*}
    (3,0), \qquad (3,1), \qquad (3,-1)
  \end{equation*}
  all have winding number $3$, as they all map to $3\in (\Z \oplus \F_3)/\Tor(\Z
  \oplus \F_3) \cong \Z$. But only $(3,0)$ is divisible by $3$ in $\Z \oplus
  \F_3$. Thus divisibility of the free quotient does \emph{not} imply
  divisibility of the original class when the same prime occurs in the torsion
  subgroup.

  More generally, if $\Omega([\alpha])$ is divisible by $q$, then we can always
  write
  \begin{equation*}
    [\alpha] = q[\gamma] + [\tau],
  \end{equation*}
  where $[\tau]$ is a torsion class. At the cocycle level this becomes
  \begin{equation*}
    \alpha = q\gamma + \delta_\Z f + \tau,
  \end{equation*}
  for a cocycle $\tau$ representing a torsion class. Note that when $q$ is
  relatively prime to the order of the torsion subgroup, the torsion term can be
  absorbed into $q\gamma$, which is exactly part \ref{prop_alg_cond:2} of
  \autoref{prop:alg_cond}.
\end{remark}

From \autoref{prop:alg_cond} we can devise an algorithm for finding a cocycle
with winding number $1$. Suppose $[\alpha]\in H^m(X; \Z)$ has nonzero image in
$H^m(X; \Z)/\Tor H^m(X; \Z)$, and choose an integral cycle $\beta\in C_m(X; \Z)$
such that its class $[\beta]\in H_m(X; \Z)$ satisfies
$\pairing{[\alpha]}{[\beta]} \neq 0$. Then, by \autoref{prop:order_of_cocycle},
$\pairing{[\alpha]}{[\beta]}$ is divisible by $\Omega([\alpha])$. In particular,
every prime divisor of $\Omega([\alpha])$ occurs among the prime divisors of
$\pairing{[\alpha]}{[\beta]}$.

If
\begin{equation}\label{eq:gcd}
  \gcd(\pairing{[\alpha]}{[\beta]}, |{\Tor H^m(X; \Z)}|) = 1,
\end{equation}
then every prime divisor of $\pairing{[\alpha]}{[\beta]}$ is relatively prime to
the order of the torsion subgroup. Thus by \autoref{prop:alg_cond}, we can
determine whether such a prime divides $\Omega([\alpha])$ by reduction modulo
that particular prime.

We note that \autoref{eq:gcd} is automatically satisfied when $H^m(X; \Z)$ is
torsion-free, and that, in particular, it is automatically satisfied for $m =
1$.

\begin{proposition}\label{prop:gamma-construction}
  Suppose $m \geq 1$ and let $\alpha \in C^m(X;\Z)$ be a cocycle and let $q$ be
  a prime such that
  \begin{equation*}
    \bar{\rho}_q ([\alpha]) = 0.
  \end{equation*}
  Then there exists $f^q \in C^{m-1}(X; \F_{q})$ satisfying
  \begin{equation*}
    \delta_{\F_q} f^q = \rho_q (\alpha).
  \end{equation*}
  For any integer lift $f \in C^{m - 1}(X; \Z)$ of $f^q$, the cochain
  \begin{equation*}
    \gamma \coloneq \frac{\alpha - \delta_\Z f}{q},
  \end{equation*}
  is an integral cocycle and satisfies
  \begin{equation}\label{eq:step-5}
    \alpha = q \gamma + \delta_\Z f.
  \end{equation}
\end{proposition}

\begin{proof}
  By definition of the induced map on cohomology, $\bar{\rho}_q ([\alpha]) = 0$
  means that
  \begin{equation*}
    [\rho_q (\alpha)] = 0 \in H^m(X; \F_q).
  \end{equation*}
  Hence, there exists a cochain $f^q \in C^{m - 1}(X; \F_q)$ such that
  \begin{equation*}
    \rho_q (\alpha) = \delta_{\F_q} f^q.
  \end{equation*}
  Take any integer lift $f$ of $f^q$. Since $\rho_q$ commutes with the
  coboundary it follows that
  \begin{align*}
    \rho_q (\alpha - \delta_\Z f)
    &= \rho_q(\alpha) - \rho_q (\delta_\Z f) \\
    &= \rho_q(\alpha) - \delta_{\F_q} \rho_q (f) \\
    &= \rho_q(\alpha) - \delta_{\F_q} f^q \\
    &= 0.
  \end{align*}
  Hence, every coefficient of $\alpha - \delta_\Z f$ is divisible by $q$. Thus,
  we can define
  \begin{equation*}
    \gamma \coloneq \frac{\alpha - \delta_\Z f}q \in C^m(X; \Z).
  \end{equation*}
  By construction,
  \begin{equation*}
    \alpha = q \gamma + \delta_\Z f.
  \end{equation*}
  Since $\alpha$ is a cocycle,
  \begin{equation*}
    \delta_\Z \alpha = q \delta_\Z \gamma + \delta_\Z^2 f = q \delta_\Z \gamma = 0.
  \end{equation*}
  As $C^{m + 1}(X; \Z)$ is free abelian, multiplication by $q$ is injective, and
  therefore,
  \begin{equation*}
    \delta_\Z \gamma = 0.
  \end{equation*}
  Hence, $\gamma$ is a cocycle.
\end{proof}

\begin{algorithm}[H]
  \caption{Reducing to a primitive cocycle}
  \label{alg:cocycle_reduction}
  \begin{algorithmic}[1]
    \Require An integral cocycle $\alpha \in C^m(X; \Z)$ whose class has nonzero
    image in $H^m(X; \Z)/\Tor H^m(X; \Z)$, and an integral cycle $\beta\in C_m(X;
    \Z)$ such that $\langle [\alpha],[\beta] \rangle \neq 0$ and
    \autoref{eq:gcd} holds.
    \Ensure An integral cocycle $\alpha_\gen$ with $\Omega([\alpha_\gen]) = 1$.

    \State $\alpha' \leftarrow \alpha$
    \State Let $P$ be the set of prime divisors of $|\langle [\alpha],[\beta] \rangle|$.

    \For{each prime $q \in P$}
      \While{$\bar{\rho}_q([\alpha']) = 0$}
        \State \parbox[t]{\linewidth}{
          Find $f^q\in C^{m - 1}(X; \F_q)$ satisfying
          \begin{equation*}
            \delta_{\F_q} f^q = \rho_q (\alpha'),
          \end{equation*}
          choose any integral lift $f$, and set
          \begin{equation*}
            \gamma = \frac{\alpha' - \delta_\Z f}{q}.
          \end{equation*}
        }
        \State $\alpha' \leftarrow \gamma$
      \EndWhile
    \EndFor

    \State $\alpha_\gen \leftarrow \alpha'$
    \State \Return $\alpha_\gen$
  \end{algorithmic}
\end{algorithm}

\begin{example}\label{ex:trefoil_alg1}
  The \texttt{Ripser} algorithm returns a prime-field cocycle that, after
  rescaling or integer linear correction, yields an integral cocycle of winding
  number $23$. This is verified using \autoref{prop:order_of_cocycle} and
  \autoref{prop:alg_cond}. Using \autoref{alg:cocycle_reduction} we successfully
  reduce the cocycle to one with winding number $1$. See
  \autoref{fig:trefoil_alg1} for a plot of the circular coordinates. See our
  \href{https://github.com/MathiasKarsrudNordal/lifting-heuristic-circular-coordinates}{GitHub}\footnote{\url{https://github.com/MathiasKarsrudNordal/lifting-heuristic-circular-coordinates}}
  for full details.
\end{example}

\begin{figure}[htbp]
  \centering
  \includegraphics[width=1.0\linewidth]{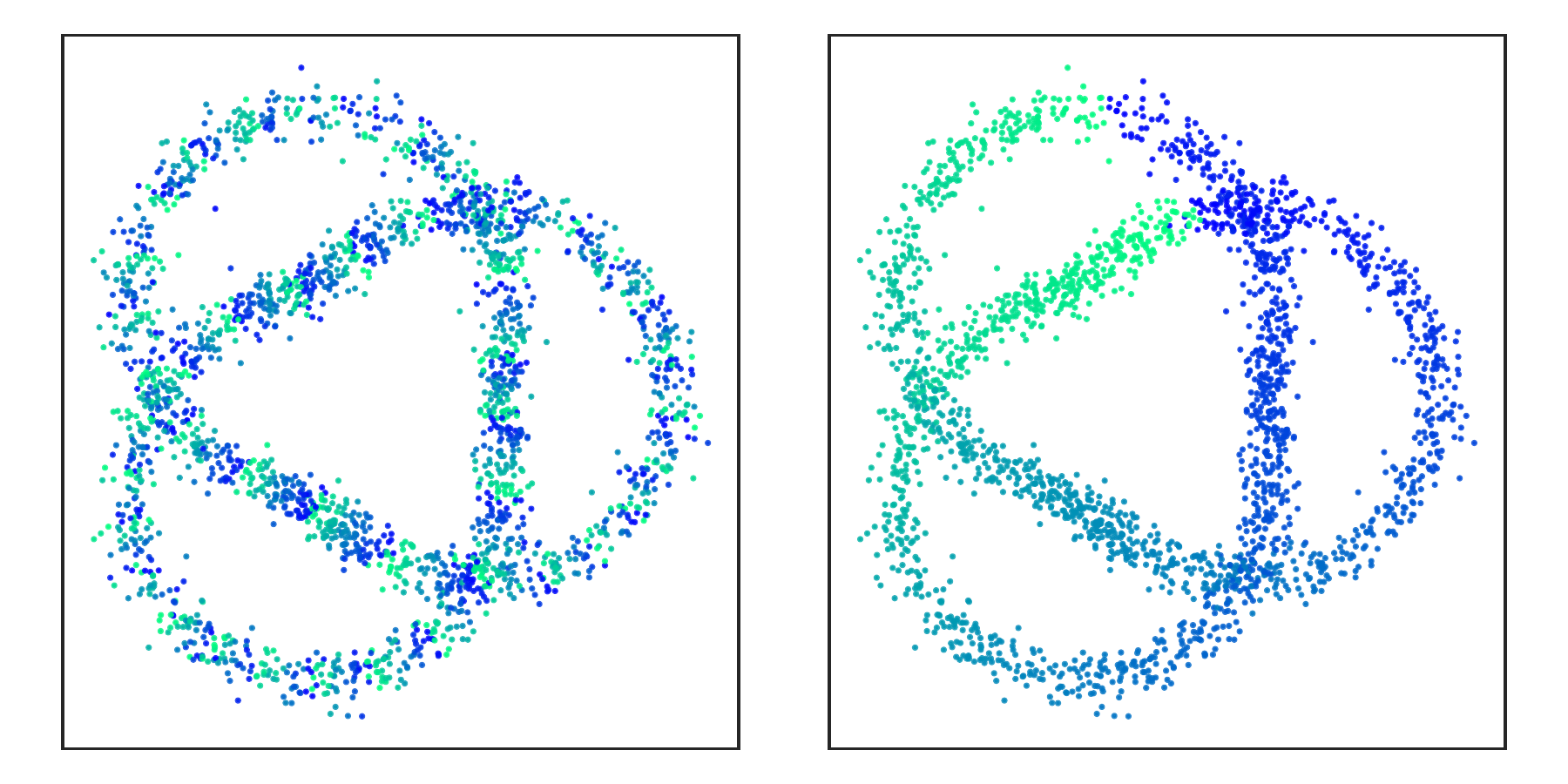}
  \caption{Circular coordinates of a cocycle with winding number $23$ (left) and
    of a cocycle with winding number $1$ (right).}
  \label{fig:trefoil_alg1}
\end{figure}

\begin{proposition}\label{prop:algorithm}
  Suppose that the inputs $\alpha$ and $\beta$ satisfy the hypotheses of
  \autoref{alg:cocycle_reduction}. In particular, $\langle [\alpha], [\beta]
  \rangle \neq 0$ and $\gcd(\langle [\alpha], [\beta]\rangle, |\Tor H^m(X; \Z)|)
  = 1$. Then \autoref{alg:cocycle_reduction} terminates and returns an integral
  cocycle $\alpha_\gen$ with $\Omega([\alpha_\gen]) = 1$. Moreover, $[\alpha] =
  \Omega([\alpha])[\alpha_\gen]$.
\end{proposition}

\begin{proof}
  Let $N\coloneq \langle [\alpha],[\beta] \rangle \neq 0$, and let $P$ be the
  set of prime divisors of $|N|$ as in \autoref{alg:cocycle_reduction}. Suppose
  that for a prime $q\in P$, the algorithm reaches a cocycle $\alpha'$
  satisfying $\bar{\rho}_q([\alpha']) = 0$. \autoref{prop:gamma-construction}
  ensures that the system in Step 5 in \autoref{alg:cocycle_reduction} always
  has a solution whenever the algorithm asks us to solve it.

  Moreover,
  \begin{equation*}
    \langle [\alpha'],[\beta] \rangle = q\langle [\gamma],[\beta] \rangle.
  \end{equation*}
  Thus each iteration of the $q$-loop divides the nonzero integer $|\langle
  [\alpha'],[\beta] \rangle|$ by $q$, and so every $q$-loop terminates. Since
  $P$ is finite, \autoref{alg:cocycle_reduction} terminates.

  We will need the following fact. Suppose that at some later stage the
  algorithm divides by a prime $r\neq q$, so that $[\alpha_{\text{old}}] = r
  [\alpha_{\text{new}}]$. Applying reduction modulo $q$ gives
  \begin{equation*}
    \bar{\rho}_q ([\alpha_{\text{old}}]) = [r]_q \bar{\rho}_q
    ([\alpha_{\text{new}}]).
  \end{equation*}
  Since $r\neq q$, $[r]_q\in \F_q$ is nonzero and thus invertible. It follows
  that
  \begin{equation}\label{eq:fact}
    \bar{\rho}_q ([\alpha_{\text{old}}]) = 0 \iff \bar{\rho}_q
    ([\alpha_{\text{new}}]) = 0.
  \end{equation}
  Thus, once the $q$-loop terminates with $\bar{\rho}_q ([\alpha'])\neq 0$, no
  subsequent division by a different prime can make the reduction modulo $q$
  vanish again.

  Let $\alpha_\gen$ be the cocycle returned by \autoref{alg:cocycle_reduction}.
  Assume, for contradiction, that
  \begin{equation*}
    \Omega([\alpha_\gen]) > 1.
  \end{equation*}
  Choose a prime $q$ such that $\Omega([\alpha_\gen])$ is divisible by $q$.
  Then, by \autoref{prop:order_of_cocycle}, $\langle [\alpha_\gen],[\beta]
  \rangle$ is divisible by $q$. Every division performed by the algorithm
  expresses the preceding class by a positive integer multiple of the following
  class. Thus $\langle [\alpha],[\beta] \rangle$ is an integer multiple of
  $\langle [\alpha_\gen],[\beta] \rangle$, and so, $N$ is divisible by $q$.
  Hence $q\in P$.

  By \autoref{eq:gcd}, $|{\Tor H^m(X; \Z)}|$ is \emph{not} divisible by $q$.
  Since $\Omega([\alpha_\gen])$ is divisible by $q$, part \ref{prop_alg_cond:2} of
  \autoref{prop:alg_cond} now gives $\bar{\rho}_q([\alpha_\gen]) = 0$. On the
  other hand, when \autoref{alg:cocycle_reduction} processed the prime $q$, its
  $q$-loop terminated only after reaching a class whose reduction modulo $q$ was
  nonzero. By \autoref{eq:fact}, all subsequent divisions by primes $r\neq q$
  preserve this nonvanishing. Thus $\bar{\rho}_q([\alpha_\gen]) \neq 0$ which is
  a contradiction. Thus $\Omega([\alpha_\gen]) = 1$.

  Finally, composing all division steps gives $[\alpha] = q_\gen [\alpha_\gen]$
  for some positive integer $q_\gen$. Passing to the torsion-free quotient gives
  \begin{equation*}
    \pi_T([\alpha]) = q_\gen \pi_T([\alpha_\gen]).
  \end{equation*}
  Since $\Omega([\alpha_\gen]) = 1$, $\pi_T([\alpha_\gen])$ is primitive. Thus the
  divisibility of $\pi_T([\alpha])$ is exactly $q_\gen$, and so, $q_\gen =
  \Omega([\alpha])$. Hence $[\alpha] = \Omega([\alpha]) [\alpha_\gen]$. This
  finishes the proof.
\end{proof}

The importance of identifying a cocycle with winding number $1$ is illustrated
in \autoref{fig:generators_matters}. In this example, we sample points from a
circle embedded in $300$-dimensional Euclidean space, i.e., $S^1 \subseteq
\R^{300}$. While PCA fails to recover the circular structure, a more
topologically aware approach such as persistent homology succeeds.

This example illustrates the importance of detecting and reducing the winding
number before interpreting the resulting circular coordinates.

\begin{figure}[htbp]
  \centering
  \includegraphics[width=1.0\textwidth]{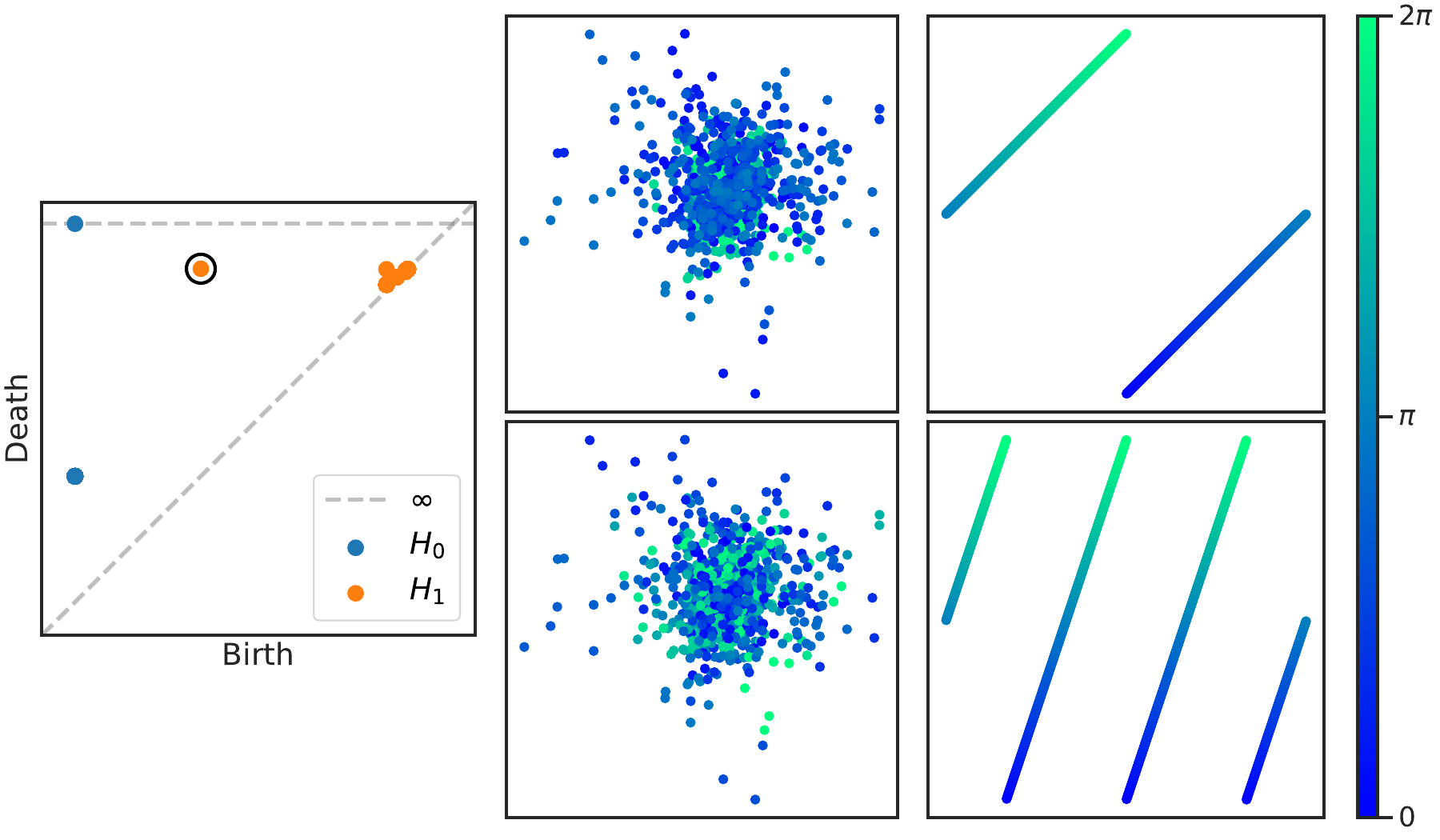}
  \caption{From the highlighted 1-dimensional feature (black circle) of the
    persistence diagram (left) we extract two different integral cocycle
    representatives: one with winding number $1$ (top row) and one with winding
    number $3$ (bottom row). The middle column shows the 2D PCA projection of
    the circle colored by the circular coordinates computed from each
    representative. The right column compares the computed circular coordinates
    ($y$-axis) against the ground-truth circular coordinates ($x$-axis). While the
    generator produces circular coordinates that align with the ground truth,
    the representative of winding number $3$ identifies points on different
    vertical lines with the same color, leading to a failure in dimensionality
    reduction.}
  \label{fig:generators_matters}
\end{figure}

\begin{remark}\label{rem:referee_prop4.4}
  The torsion complication disappears completely for $m = 1$, i.e., the
  motivating circular-coordinate problem, as $H^1(X; \Z)$ is torsion-free. Thus
  in that setting, we have
  \begin{equation*}
    q \mid \Omega([\alpha]) \iff \bar{\rho}_q([\alpha]) = 0
  \end{equation*}
  and \autoref{alg:cocycle_reduction} works without the additional coprimality
  hypothesis.
\end{remark}

\autoref{prop:gamma-construction} shows that Step 5 can be performed by solving
a coboundary equation over $\F_q$, lifting the resulting cochain arbitrarily to
the integers, and then dividing coefficientwise by $q$. In degree $1$, both this
computation and the construction of the cycle used in the Kronecker pairing
admit a particularly simple spanning-forest implementation.

\subsubsection*{Implementation of \autoref{alg:cocycle_reduction}}
In the motivating case of circular coordinates, where $m = 1$, the computations
required in \autoref{alg:cocycle_reduction} admit a particularly simple
implementation using a spanning forest of the $1$-skeleton of $X$. We describe
two uses of this construction.

First, let $q$ be one of the primes considered in
\autoref{alg:cocycle_reduction}. To determine whether
\begin{equation*}
  \bar{\rho}_q ([\alpha]) = 0,
\end{equation*}
choose a spanning tree in each connected component of the $1$-skeleton and fix
one root vertex in each component. Define a $0$-cochain
\begin{equation*}
  f^q \in C^0(X; \F_q)
\end{equation*}
by assigning the value $0$ to each root and integrating $\rho_q(\alpha)$ along
the edges of the corresponding spanning tree. Orient each tree edge away from
its root, extending a cochain to the opposite orientation by $\alpha(vu) =
-\alpha(uv)$. If $u$ is the parent of $v$ and $f^q(u)$ has already been defined,
set $f^q(v) = f^q(u) + \rho_q(\alpha)(uv)$. The resulting $f^q$ satisfies
\begin{equation*}
  \delta_{\F_q} f^q = \rho_q (\alpha)
\end{equation*}
on every tree edge. The class $\bar{\rho}_q ([\alpha])$ vanishes if and only if
the same equality holds on every non-tree edge. Thus the same spanning-forest
computation both determines whether the condition in Step~4 of
\autoref{alg:cocycle_reduction} holds and, when it does, produces the cochain
$f^q$ required in \autoref{prop:gamma-construction}. Consequently, Step~5 can
then be carried out by choosing an integer lift $f$ of $f^q$ and setting
\begin{equation*}
  \gamma = \frac{\alpha - \delta_\Z f}{q}.
\end{equation*}

Second, the same spanning forest can be used to construct the integral cycle
required for the Kronecker pairing. Let $\alpha \in C^1(X; \Z)$ be a cocycle
with $[\alpha]\neq 0$. Choose a spanning forest as above and construct $h\in
C^0(X; \Z)$ by integrating $\alpha$ along its tree edges, so that
\begin{equation*}
  \delta_\Z h(\kappa) = \alpha(\kappa)
\end{equation*}
for every tree edge $\kappa$. Let $\kappa = uv$ be a non-tree edge, and let
$P_\kappa$ denote the unique oriented path in the spanning tree from $u$ to $v$.
Then
\begin{equation*}
  \beta_\kappa \coloneq \kappa - P_\kappa
\end{equation*}
is an integral $1$-cycle. Indeed,
\begin{equation*}
  \partial_\Z \beta_\kappa = \partial_\Z \kappa - \partial_\Z P_\kappa =
  (v - u) - (v - u) = 0.
\end{equation*}
Moreover,
\begin{align*}
  \pairing{\alpha}{\beta_\kappa}
  &= \pairing{\alpha}{\kappa - P_\kappa} \\
  &= \pairing{\alpha}{\kappa} - \pairing{\alpha}{P_\kappa} \\
  &= \alpha(\kappa) - \pairing{\delta_\Z h}{P_\kappa} \\
  &= \alpha(\kappa) - \pairing{h}{\partial_\Z P_\kappa} \\
  &= \alpha(\kappa) - \pairing{h}{v - u} \\
  &= \alpha(\kappa) - (\delta_\Z h) (\kappa).
\end{align*}
If this pairing were zero for every non-tree edge, then $\alpha = \delta_\Z h$
on every edge of the $1$-skeleton, and hence $[\alpha] = 0$, contradicting our
assumption. Therefore, there exists a non-tree edge $\kappa$ for which
\begin{equation*}
  \pairing{[\alpha]}{[\beta_\kappa]} = \pairing{\alpha}{\beta_\kappa} \neq 0.
\end{equation*}
Thus, in degree $1$, both the cycle required for the Kronecker pairing and the
prime-field cochain required in the division step can be obtained directly from
a spanning forest of the $1$-skeleton.

\begin{remark}
  \autoref{prop:gamma-construction} requires no torsion hypothesis. Once
  $\bar{\rho}_q([\alpha])=0$ has been established, the division step is valid in
  any degree. What changes in higher degrees is the computation of the auxiliary
  objects. For $m > 1$, finding $f^q$ amounts in general to solving the
  prime-field linear system
  \begin{equation*}
    \delta_{\F_q} f^q = \rho_q(\alpha),
  \end{equation*}
  and there is no comparably simple graph-theoretic spanning-tree construction
  in general. Such cycles may instead be obtained using homological linear
  algebra or representative-cycle methods, for example methods in the spirit of
  \cite{Virk23}. These methods typically produce cycle representatives over a
  prime field, whereas \autoref{alg:cocycle_reduction} requires an integral
  cycle in order to evaluate the Kronecker pairing over $\Z$. This motivates the
  next section, where we study when a prime-field cycle representative can be
  lifted coefficientwise to an integral cycle.
\end{remark}

\subsection{Lifting homology classes}\label{subsec:lifting_homology}
The procedure in \autoref{subsec:order_cocycle} requires an integral cycle in
order to evaluate the Kronecker pairing. In practice, persistent homology
computations often provide cycle representatives with coefficients in a prime
field. For this reason we record the homological analogue of the coefficientwise
lifting criterion from \autoref{sec:lifting}.

We may ask whether a homology class over $\F_p$ admits any lift to integral
homology. This is governed by the Bockstein long exact sequence.

We use the reduction map $\rho_p$ and balanced lift $\lambda_p$ as introduced in
\autoref{sec:prelim}, and we assume that $m$ is an integer greater than or equal
to $1$ throughout this section.

The sequence
\begin{equation*}
  \begin{tikzcd}[column sep = scriptsize]
    {0} & {\Z} & {\Z} & {\F_p} & {0}
    \arrow[from=1-1, to=1-2]
    \arrow["p", from=1-2, to=1-3]
    \arrow[from=1-3, to=1-4]
    \arrow[from=1-4, to=1-5]
  \end{tikzcd}
\end{equation*}
induces the long exact sequence in homology
\begin{equation*}
  \begin{tikzcd}[column sep = scriptsize]
    {\cdots} &
    {H_m(X; \Z)} &
    {H_m(X; \F_p)} &
    {H_{m - 1}(X; \Z)} &
    {H_{m - 1}(X; \Z)} &
    {\cdots.}
    \arrow[from=1-1, to=1-2]
    \arrow["\bar{\rho}_p", from=1-2, to=1-3]
    \arrow["\beta", from=1-3, to=1-4]
    \arrow["p", from=1-4, to=1-5]
    \arrow[from=1-5, to=1-6]
  \end{tikzcd}
\end{equation*}
Hence,
\begin{equation*}
  \bar{\rho}_p \colon H_m(X; \Z) \to H_m(X; \F_p)
\end{equation*}
is surjective whenever $H_{m - 1}(X; \Z)$ has no $p$-torsion. In particular,
when $m = 1$, the map is always surjective as $H_0(X; \Z)$ is torsion-free. This
gives an existence criterion for integral homology classes. We note, however,
that it does \emph{not} imply that the balanced coefficientwise lift of a chosen
cycle representative is a cycle.

We now give a sufficient condition for this to be true. Let
\begin{equation*}
  c^p = \sum_{\sigma \in X^{(m)}} c_\sigma \sigma \in C_m(X; \F_p)
\end{equation*}
be an $m$-cycle. For an oriented $(m - 1)$-simplex $\tau$, we define
\begin{equation*}
  I_\tau(c^p) \coloneq \{\sigma\in \suppp(c^p) \mid \tau \text{ is a face of }
  \sigma\}, \qquad d_\tau \coloneq |I_\tau (c^p)|.
\end{equation*}
For an oriented $m$-simplex $\sigma$ containing $\tau$, let
$\incidence{\sigma}{\tau}\in \{-1,1\}$ denote the corresponding incidence
number. Note that since $c^p$ is a cycle, we have
\begin{equation}\label{eq:coeff_vanish_boundary}
  \sum_{\sigma\in I_\tau(c^p)} \incidence{\sigma}{\tau} c_\sigma = 0
\end{equation}
in $\F_p$.

\begin{proposition}\label{prop:lift_cycle_is_cycle}
  Let $c^p\in C_m(X; \F_p)$ be an $m$-cycle, $r\in \F_p^\ast$ and define
  \begin{equation*}
    \tilde{c}_r \coloneq \lambda_p(rc^p) = \sum_{\sigma\in \suppp(c^p)}
    \lambda_p(rc_\sigma)\sigma \in C_m(X; \Z).
  \end{equation*}
  Suppose that for every $(m - 1)$-simplex $\tau$,
  \begin{equation}\label{eq:cycle_bound}
    \sum_{\sigma\in I_\tau(c^p)} |\lambda_p(rc_\sigma)| < p.
  \end{equation}
  Then $\tilde{c}_r$ is an integral $m$-cycle. Furthermore, if $s\in \Z$ is any
  integer satisfying $[s]_p = r^{-1}$, then $c\coloneq s\tilde{c}_r$ is an
  integral $m$-cycle satisfying $\rho_p(c) = c^p$.
\end{proposition}

\begin{proof}
  Fix an $(m - 1)$-simplex $\tau$. The coefficient of $\tau$ in the integral
  boundary $\partial \tilde{c}_r$ is
  \begin{equation*}
    S_\tau \coloneq \sum_{\sigma\in I_\tau(c^p)} \incidence{\sigma}{\tau}
    \lambda_p(rc_\sigma).
  \end{equation*}
  By reducing this integer modulo $p$, we get
  \begin{align*}
    S_\tau &\equiv \sum_{\sigma\in I_\tau(c^p)} \incidence{\sigma}{\tau}
             \lambda_p(rc_\sigma) \modd p\\
    &= r \sum_{\sigma\in I_\tau(c^p)} \incidence{\sigma}{\tau} c_\sigma\\
    &= 0 \modd p,
  \end{align*}
  where the last equality follows from \autoref{eq:coeff_vanish_boundary}. Thus
  $S_\tau\in p\Z$.

  By the triangle inequality and \autoref{eq:cycle_bound}, we have
  \begin{equation*}
    |S_\tau| \leq \sum_{\sigma\in I_\tau(c^p)} |\lambda_p(rc_\sigma)| < p.
  \end{equation*}
  Since the only multiple of $p$ having absolute value strictly smaller than $p$
  is $0$, we have $S_\tau = 0$. Moreover, since this holds for every $(m -
  1)$-simplex $\tau$, $\partial \tilde{c}_r = 0$. Thus $\tilde{c}_r$ is an
  integral $m$-cycle.

  Finally, as $\tilde{c}_r$ is an integral $m$-cycle, so is $c = s\tilde{c}_r$.
  Moreover,
  \begin{equation*}
    \rho_p(c) = [s]_p \rho_p(\tilde{c}_r) = r^{-1} (rc^p) = c^p.
  \end{equation*}
  This finishes the proof.
\end{proof}

\begin{corollary}\label{cor:lift_cycle_is_cycle}
  Under the notation and hypotheses of \autoref{prop:lift_cycle_is_cycle}, if
  \begin{equation*}
    |\lambda_p(rc_\sigma)| \leq \floor{p - 1}{d_\tau}
  \end{equation*}
  for every $(m - 1)$-simplex $\tau$ with $d_\tau > 0$ and every $\sigma\in
  I_\tau (c^p)$, then $\tilde{c}_r$ is an integral $m$-cycle.
\end{corollary}

\begin{proof}
  For each $\tau$ with $d_\tau > 0$, we have
  \begin{equation*}
    \sum_{\sigma\in I_\tau(c^p)} |\lambda_p(rc_\sigma)| \leq d_\tau
    \floor{p - 1}{d_\tau} \leq p - 1 < p,
  \end{equation*}
  and so the corollary follows by applying \autoref{prop:lift_cycle_is_cycle}.
\end{proof}

\begin{remark}
  Suppose that $c^p$ is a $1$-cycle whose support is an embedded simplicial
  circle. Then every vertex in the support is incident to exactly two support
  edges, i.e., $d_\tau = 2$. Moreover, since every balanced coefficient,
  $\lambda_p(c_\sigma)$, has absolute value at most $(p - 1)/2$, we have
  \begin{equation*}
    \sum_{\sigma\in I_\tau (c^p)} |\lambda_p(c_\sigma)| \leq p - 1 < p.
  \end{equation*}
  Equivalently, at each vertex $\tau$, the two incidence-weighted coefficients
  $\incidence{\sigma}{\tau} c_\sigma$ are negatives of one another in $\F_p$.
  Thus their balanced representatives cancel over $\Z$.
\end{remark}

\autoref{prop:lift_cycle_is_cycle} may be combined directly with
\autoref{thm:main}. To see this, let $n \coloneq |\suppp(c^p)|$ and
\begin{equation*}
  M\coloneq \max_{\tau\in X^{(m - 1)}} d_\tau.
\end{equation*}
Note that $M$ is the largest number of simplices in the support of $c^p$ that
contain a common $(m - 1)$-face.

\begin{corollary}\label{cor:cycle_large_prime}
  Let $c^p\in C_m(X; \F_p)$ be a nonzero cycle. If $p - 1 > M^n$, then there
  exists $r\in \F_p^\ast$ such that $\lambda_p(rc^p)$ is an integral cycle.
  Furthermore, if $s\in \Z$ satisfies $[s]_p = r^{-1}$, then $s\lambda_p(rc^p)$
  is an integral cycle lift of $c^p$.
\end{corollary}

\begin{proof}
  Note that $c^p\neq 0$ implies $M\geq 1$. Apply \autoref{thm:main} to the
  vector of the $n$ nonzero coefficients of $c^p$, with $k = M$. Then, under the
  assumption that $p - 1 > M^n$, there exists $r\in \F_p^\ast$ such that
  \begin{equation*}
    |\lambda_p(rc_\sigma)|\leq \floor{p - 1}M
  \end{equation*}
  for every $\sigma\in \suppp(c^p)$.

  For every $\tau$ with $d_\tau > 0$, we have $d_\tau\leq M$, and so,
  \begin{equation*}
    \floor{p - 1}M \leq \floor{p - 1}{d_\tau}.
  \end{equation*}
  Thus, by \autoref{cor:lift_cycle_is_cycle}, $\lambda_p(rc^p)$ is an integral
  cycle. Moreover, scaling by an integer representative of $r^{-1}$ then gives
  an integral cycle reducing to $c^p$.
\end{proof}

\begin{remark}
  The difference between the cocycle and cycle criteria is combinatorial. For an
  $m$-cocycle, evaluating the coboundary of an $(m + 1)$-simplex produces a
  relation involving exactly $m + 2$ coefficients. On the other hand, for an
  $m$-cycle, the coefficient of a fixed $(m - 1)$-simplex in the boundary
  receives contributions from all $m$-simplices in the support containing that
  face. In other words, the relevant parameter is $d_\tau$ locally, or $M$
  globally, rather than simply the degree $m$.

  The bound $p - 1 > M^n$ should be viewed as a theoretical sufficient
  condition. As in \autoref{sec:lifting}, it does \emph{not} by itself establish
  that using such a prime is computationally preferable to general integer
  linear algebra methods.
\end{remark}

\end{document}